\documentclass[12pt,a4paper,oneside,english,keywords]{amsart}
\usepackage[T1]{fontenc}
\usepackage[latin9]{inputenc}
\usepackage{fancyhdr}
\usepackage{color}
\usepackage{mathrsfs}
\usepackage{url}
\usepackage{amsthm}
\usepackage{amssymb}
\usepackage{graphicx}
\usepackage{setspace}
\makeatletter

\pdfpageheight\paperheight
\pdfpagewidth\paperwidth

\numberwithin{equation}{section}
\numberwithin{figure}{section}
\theoremstyle{plain}
\newtheorem{thm}{\protect\theoremname}
\theoremstyle{definition}
\newtheorem{defn}[thm]{\protect\definitionname}
\theoremstyle{definition}
\newtheorem*{example*}{\protect\examplename}
\theoremstyle{plain}
\newtheorem{prop}[thm]{\protect\propositionname}
\theoremstyle{plain}
\newtheorem{lem}[thm]{\protect\lemmaname}
\theoremstyle{plain}
\newtheorem{cor}[thm]{\protect\corollaryname}
\theoremstyle{definition}
\newtheorem{example}[thm]{\protect\examplename}

\makeatletter
\def\paragraph{\@startsection{paragraph}{4}%
  \z@\z@{-\fontdimen2\font}%
  {\normalfont\bfseries}}
\makeatother
\makeatother

\usepackage{babel}
\usepackage{listings}
\providecommand{\corollaryname}{Corollary}
\providecommand{\definitionname}{Definition}
\providecommand{\examplename}{Example}
\providecommand{\lemmaname}{Lemma}
\providecommand{\propositionname}{Proposition}
\providecommand{\theoremname}{Theorem}

\begin{document}
\title[Bases of $U_{q}\left(\mathfrak{g}\right)$'s in Lyndon words]{Bases of quantum group algebras in terms of Lyndon words}
\author{Eremey Valetov}
\date{September 22, 2016}
\begin{abstract}
We have reviewed some results on quantized shuffling, and in particular,
the grading and structure of this algebra. In parallel, we have summarized
certain details about classical shuffle algebras, including Lyndon
words (primes) and the construction of bases of classical shuffle
algebras in terms of Lyndon words. We have explained how to adapt
this theory to the construction of bases of quantum group algebras
in terms of Lyndon words. This method has a limited application to
the specific case of the quantum group parameter being a root of unity,
with the requirement that specialization to the root of unity is non-restricted.
As an additional, applied part of this work, we have implemented a
Wolfram Mathematica package with functions for quantum shuffle multiplication
and constructions of bases in terms of Lyndon words.
\end{abstract}

\address{Faculté de Mathématiques, Université Pierre \& Marie Curie (Paris
VI), 4 place Jussieu, 75252 Paris CEDEX 05, France. Also at Michigan
State University, East Lansing, MI 48824, USA, and at Institut für
Kernphysik, Forschungszentrum Jülich, Germany.}

\maketitle
\tableofcontents{}

\section{Introduction}

In \cite{Radford}, D. Radford has developed the classical shuffle
algebra, proposed a method of constructing its bases in terms of Lyndon
words (primes), and explained how this can be applied to commutative
pointed irreducible Hopf algebras.

A common object of shuffle algebras is a tensor space $T\left(V\right)=\oplus_{n=0}^{+\infty}V^{\otimes n}$,
where $V$ is a module. Consider an element $v_{a}=v_{x_{1}}\otimes\cdots\otimes v_{x_{n}}$
of this tensor space. Let $\left\{ v_{x}|x\in X\right\} $ be an indexed
basis of $V$ and $S=\left(X\right)$ the free semigroup generated
by $X$. We establish a notation for the $a$ in $v_{a}$ that is
bijective to the vector $\left(x_{1},\cdots,x_{n}\right)$ by setting
$a=x_{1}\cdots x_{n}\in S$. This $a$ is called a word, and may be
viewed as a nonempty string of characters.

If the set $X$ is totally ordered with some relation $\le$, we can
extend it to $S$ as a lexicographic order. Lexicographic order has
multiple applications; for example, a variant of lexicographical order
is used the set of real numbers $\mathbb{R}$ in decimal notation,
and strings in computer science are commonly lexicographically compared.

In 1954, R. Lyndon has investigated a special type of words called
standard lexicographic sequences or regular words \cite{lyndon},
which were later also named Lyndon words and primes. If $a\in S$,
then $a$ is called a Lyndon word if for any its factorization $uv=a$
with $u,v\in S$ we have $v<a$. Any word in $S$ can be factorized
into Lyndon words, and this factorization is unique.

Shuffle algebras use a special multiplication operation called shuffle
multiplication. Shuffle multiplication of $v_{a}$ and $v_{b}$, where
$a,b\in S$, is defined as the sum of elements of the tensor space
that correspond to a special type of permutation that may be called
a $\left(m_{0},\cdots,m_{r}\right)$-shuffle applied to $a$ and $b$.
We will cover this in detail later on.

To construct bases of $T\left(V\right)$ in terms of Lyndon words,
we define $X_{a}$, which is a shuffle multiplication of elements
corresponding to Lyndon words in the unique prime factorization of
$a$. The $X_{a}$'s form a linear basis of $T\left(V\right)$.

M. Rosso has introduced the quantum shuffle algebra \cite{MRosso},
which can be viewed as a quantized version of the classical shuffle
algebra. Rosso described the construction of its bases in terms of
Lyndon words. A quantum shuffle algebra is constructed from the cotensor
Hopf algebra of a Hopf bimodule, and this applies to the Hopf subalgebra
$U_{q}^{+}\left(\mathfrak{g}\right)$ of a quantum group algebra $U_{q}\left(\mathfrak{g}\right)$.

In the case of quantum shuffle algebras, the quantum deformation manifests
itself through permutations of the form $s\left(v\otimes w\right)=w\otimes v$
being replaced by diagonal braiding of the form $\sigma\left(v\otimes w\right)=q_{ij}\left(w\otimes v\right)$
in the implementation of $\left(m_{0},\cdots,m_{r}\right)$-shuffles.

In the same paper \cite{MRosso}, Rosso has discussed how the quantum
shuffle algebra theory can be applied to the Hopf bimodules $U_{q}^{+}\left(\mathfrak{g}\right)$.
This application may not be immediately obvious to a reader.

We would like to illuminate the path of familiarization with the method
of bases construction in terms of Lyndon words for classical shuffle
algebras, quantum shuffle algebras, and how this method can be applied
to quantum group algebras. Along the way, we wish to implement quantum
shuffle multiplication and the corresponding method of bases construction
as a Wolfram Mathematica package (function library).

We recognize that an important issue in research is setting priorities
and specializing in a number of chosen subject matter areas. With
shuffle algebras being the subject matter of specialization in this
work, we honor this concept by deriving each proof we provide in this
paper on our own. On the level of formulations of mathematical propositions,
our aim is to preserve precision and legacy.

We will use the terms 'Lyndon word' and 'prime' interchangeably based
on our perception of their common usage in various context. In the
broad picture, we prefer to use the term 'Lyndon word'\footnote{A trivial paragaph intended for internal audience was deleted here.}.

\section{Literature review}

We consider the papers by Radford and Rosso \cite{Radford,MRosso}
mentioned in the introduction to be the primary references for our
work, and we will now briefly outline their contents.

In ``A natural ring basis for the shuffle algebra and an application
to group schemes'' \cite{Radford}, Radford (1) defines the shuffle
algebra of a set $X$, including the shuffle multiplication; (2) defines
$U$-graded algebra (coalgebra), with $u\left(x\right)$ as the number
of $x$'s in the factorization of $a\in S$; (3) shows how commutative
pointed irreducible Hopf algebras may be embedded into classical shuffle
algebras; (4) details the structure of the Lyndon words of $S$; (5)
discusses prime factorization and shuffles; (6) presents a method
of constructing bases of classical shuffle algebras in terms of Lyndon
words; and (7) covers some applications and other details.

In ``Quantum groups and quantum shuffles'' \cite{MRosso}, Rosso
(1) discusses the Hopf bimodule structure; (2) specifies a braiding
in the category of Hopf bimodules that was introduced by Woronowicz
\cite{woronowicz}; (3) constructs the cotensor Hopf algebra from
a Hopf bimodule, including the quantum shuffle multiplication; (4)
shows an embedding of a tensor space $T\left(V\right)$ in the cotensor
Hopf algebra; (5) introduces the quantum symmetric algebra that is
a sub-Hopf algebra of the cotensor Hopf algebras; (6) details the
universal construction of a quantum shuffle algebra in the braid category;
(7) gives examples from abelian group algebras; (8) details the structure
of Lyndon words and a method of bases construction of quantum shuffle
algebras in terms of Lyndon words; (9) covers consequences of growth
conditions; and (10) presents a theory for and the applications of
the inductive construction of higher rank quantized enveloping algebras.

M. Sweedler's ``Hopf algebras'' \cite{sweedler} is useful as a
review of Hopf algebras in that it provides thorough coverage of relevant
concepts such as comodules and coinvariants. ``Quantum groups''
\cite{kassel} by C. Kassel introduces crossed modules and modules
over the quantum (or, ``Drinfel'd'') double. Additional references
on the quantum double are ``Quantum groups'' \cite{drinfeld} by
V. Drinfel'd and ``Doubles of quasitriangular Hopf algebras'' \cite{majid}
by S. Majid. For notations and facts about the symmetric group, one
may refer to ``Groupes et algèbres de Lie'' \cite{bourbaki} by
N. Bourbaki. ``Braid groups'' \cite{braidgroups} by C. Kassel and
V. Turaev is a good introduction to the Artin braid groups, which
is relevant to the study of quantum shuffle algebras.

Key references used in \cite{MRosso} by Rosso are ``Differential
calculus on compact matrix pseudogroups (quantum groups)'' \cite{woronowicz}
by S. Woronowicz and ``Quantum groups and representations of monoidal
categories'' \cite{yetter} by D. Yetter. We found the material in
\cite{woronowicz} regarding braidings and braid equations to be illuminating.
In \cite{yetter}, the discussion of Hopf algebra structure, including
comodules and bimodules, somewhat parallels that in \cite{MRosso}.
The diagrammatic notation in \cite{yetter} is a real gem that enables
the reader to visually follow through Hopf algebra calculations and
evolutions.

For the purposes of coherency, rigor, and broad familiarization with
the subject matter, we have reviewed a number of papers on quantum
shuffle algebras, including ``Quantum quasi-shuffle algebras'' \cite{jian}
by R.-Q. Jian et al., ``Quantum symmetric algebras'' \cite{dechela1}
and ``Quantum symmetric algebras II'' \cite{dechela2} by D. de
Chela and J. Green, and ``Dual canonical bases, quantum shuffles
and $q$-characters'' by B. Leclerc.

Having understood classical and quantum shuffle algebras, we have
turned our attention to quantum group algebras. The following landmark
textbooks contain a wealth of information about quantum groups, including
basic definitions, universal $R$-matrices, and braidings: ``A guide
to quantum groups'' \cite{charipressley} by V. Chari and A. Pressley,
``Introduction to quantum groups'' \cite{lusztig} by G. Lusztig,
``Lectures on quantum groups'' \cite{jantzen} by J. Jantzen, and
``Quantum groups'' \cite{kassel} by C. Kassel.

The paper ``A formula for the $R$-matrix using a system of weight
preserving endomorphism'' \cite{tingley} by P. Tingley was useful
in finding the form of the standard universal $R$-matrix of a quantum
group algebra. While in general, the universal $R$-matrix is not
unique, ``The uniqueness theorem for the universal $R$-matrix''
\cite{khoroshin} by S. Khoroshin and V. Tolstoy is a relevant and
instructive text.

Next, we have researched the case of quantum groups with the quantum
parameter $q$ as a root of unity. We have used the authoritative
details in \cite{charipressley} as the main reference for that purpose.
``Quantum groups at roots of 1'' \cite{Lusztig1990} by G. Lusztig
is a historical reference on that topic. The study aid ``Quantum
groups at root of unity'' \cite{singh} by B. Singh is a concise
overview of the root of unity case.

We have used \cite{lusztig} and ``On the automorphisms of $U_{q}^{+}\left(\mathfrak{g}\right)$''
\cite{dumas} by N. Andruskiewitsch and F. Dumas as references to
identify that the positive part $U_{q}^{+}\left(\mathfrak{g}\right)$
of a quantum group algebra $U_{q}\left(\mathfrak{g}\right)$ is a
Nichols algebra with diagonal braiding. In ``A survey on Nichols
algebras'' \cite{takeuchi}, M. Takeuchi reviews the categorical
approach to Nichols algebras and their braided shuffle algebras aspect.
We returned to \cite{MRosso} to confirm and validate the conclusions
on application of the method of bases construction in terms of Lyndon
words to quantum group algebras.

\section{Quantum group algebras I}

\subsection{Quantum group algebras}

Since the primary objective of this work is to apply the shuffle algebra
theory to quantum group algebras, we will begin by defining quantum
group algebras and referring to the braiding relations of tensor products
of their modules. To some extent, there appears to be a convention
to provide a detailed definition of a quantum group algebra in papers
on this topic. The purpose of this is not only to remind the reader
of the relatively numerous relations, but also to advise about the
version or variation of the notation to be used in the specific paper.
For example, the generators denoted here as $e_{i}$ and $f_{i}$
are also referred to in some papers as $X_{i}^{+}$ and $X_{i}^{-}$
respectively.
\begin{defn}[Quantum group algebra -- in principle]
There are numerous related definitions of quantum group algebras.
V. Drinfel'd \cite{drinfeld} and M. Jimbo have formalized the definition
of a quantum group algebra as a Hopf algebra that is a deformation
of the universal enveloping algebra of a Kac-Moody algebra.
\end{defn}

Historically, the first application of quantum group algebras was
the quantum inverse scattering method in statistical mechanics in
the first half of the 1980s. Other applications include probability
theory, harmonic analysis, and number theory \cite{TwentyFive}.

\subsection{Drinfel'd-Jimbo type quantum group algebras}
\begin{defn}[Quantum group algebra -- formal\footnote{This definition was obtained from the Wikipedia page ``Quantum group''
at \url{https://en.wikipedia.org/wiki/Quantum_group} and carefully
compared with \cite{sam} and numerous other sources.}]
Let $A=\left(a_{ij}\right)$ be the Cartan matrix of a Kac-Moody
algebra and let $q\in\mathbb{C}$ such that $q\notin\left\{ 0,1\right\} $.
Let $\mathfrak{g}$ be a Kac-Moody algebra with the Cartan matrix
$A$. Then the quantum group $U_{q}\left(\frak{g}\right)$ is defined
as the unital associative algebra with Chevalley generators $k_{\lambda}$,
$e_{i}$, and $f_{i}$ as follows:

\begin{itemize}
\item $k_{0}=1$, $k_{\lambda}k_{\mu}=k_{\lambda+\mu}$
\item $k_{\lambda}e_{i}k_{\lambda}^{-1}=q^{\left(\lambda,\alpha_{i}\right)}e_{i}$,
$k_{\lambda}f_{i}k_{\lambda}^{-1}=q^{-\left(\lambda,\alpha_{i}\right)}f_{i}$
\item $\left[e_{i},f_{j}\right]=\delta_{ij}\frac{k_{i}-k_{i}^{-1}}{q_{i}-q_{j}^{-1}}$
\item if $i\ne j$, then $\sum_{n=0}^{1-a_{ij}}\left(-1\right)^{n}$$\frac{\left[1-a_{ij}\right]_{q_{i}}!}{\left[1-a_{ij}-n\right]_{q_{i}}!\left[n\right]_{q_{i}}!}e_{i}^{n}e_{j}e_{i}^{1-a_{ij}-n}=0$
and $\sum_{n=0}^{1-a_{ij}}\left(-1\right)^{n}$$\frac{\left[1-a_{ij}\right]_{q_{i}}!}{\left[1-a_{ij}-n\right]_{q_{i}}!\left[n\right]_{q_{i}}!}f_{i}^{n}f_{j}f_{i}^{1-a_{ij}-n}=0$
($q$-Serre relations)
\end{itemize}
where $\lambda$ is an element of the weight lattice, $\alpha_{i}$
are the simple roots, $\left(\,,\,\right)$ is an invariant symmetric
bilinear form, $k_{i}=k_{\alpha_{i}}$, $d_{i}=\left(\alpha_{i},\alpha_{i}\right)/2$
(making $B=\left(d_{i}^{-1}a_{ij}\right)$ symmetric), $q_{i}=q^{d_{i}}$,
$\left[0\right]_{q_{i}!}=1$, $\left[n\right]_{q_{i}}!=\prod_{m=1}^{n}\left[m\right]_{q_{i}}$
for all $n\in\mathbb{N}$ ($q$-factorial), and $\left[m\right]_{q_{i}}=\frac{q_{i}^{m}-q_{i}^{-m}}{q_{i}-q_{i}^{-1}}$
($q$-number).

To make it a Hopf algebra, we can define the counit as $\epsilon\left(k_{\lambda}\right)=1$
and $\epsilon\left(e_{i}\right)=\epsilon\left(f_{i}\right)=0$ and
select a compatible coproduct $\triangle$ and an antipode $S$.

\end{defn}

\begin{example*}[Common definition of coproduct]
{\scriptsize{}}The coproduct of a quantum group algebra $U_{q}\left(\frak{g}\right)$
is often defined by

\begin{itemize}
\item $\triangle\left(k_{\lambda}\right)=k_{\lambda}\otimes k_{\lambda}$
\item $\triangle\left(e_{i}\right)=1\otimes e_{i}+e_{i}\otimes k_{i}$
\item $\triangle\left(f_{i}\right)=k_{i}^{-1}\otimes f_{i}+f_{i}\otimes1$
\end{itemize}
Other definitions are possible and used.

\end{example*}

\paragraph*{The limit of $q\rightarrow1$}

In the $q\rightarrow1$ limit, the quantum group algebra $U_{q}\left(\frak{g}\right)$
relations approach universal enveloping algebra $U\left(\frak{g}\right)$
relations, with
\begin{itemize}
\item $k_{\lambda}\rightarrow1$
\item $\frac{k_{\lambda}-k_{-\lambda}}{q_{i}-q_{i}^{-1}}\rightarrow t_{\lambda}$
\end{itemize}
where $t_{\lambda}$ are elements of the Cartan subalgebra and $\left(t_{\lambda},h\right)=\lambda\left(h\right)$
for all elements $h$ in the Cartan subalgebra.

\paragraph*{Triangular decomposition}

Quantum group algebra $U_{q}\left(\mathfrak{g}\right)$ has a triangular
decomposition into negative $U^{+}$, neutral $U^{0}$, and positive
$U^{+}$ parts \cite[p. 109]{takeuchi}:

\[
U_{q}\left(\mathfrak{g}\right)=U^{-}\otimes U^{0}\otimes U^{+}
\]

\subsection{Representations of $U_{q}\left(\mathfrak{g}\right)$}

\paragraph*{Isomorphism of $U_{q}\left(\mathfrak{g}\right)$-modules {\footnotesize{}\cite{jantzen,charipressley,lusztig}}}

While for $U\left(\mathfrak{g}\right)$-modules $V$ and $W$ the
functorial isomorphism between $V\otimes W$ and $W\otimes V$ is
a flip map $P:v\otimes w\mapsto w\otimes v$, for $U_{q}\left(\mathfrak{g}\right)$
it is, in general, not an isomorphism. However, $U_{q}\left(\mathfrak{g}\right)$
is nearly quasitriangular, i.e. there exists an infinite formal sum
that plays the role of an $R$-matrix, which may be called the quasi
$R$-matrix \cite{lusztig}. Commonly, an alternative construction
called universal $R$-matrix is used \cite{Tingley_constructingthe,charipressley}.
A functorial isomorphism $R_{V,W}$ between $V\otimes W$ and $W\otimes V$
for $U_{q}\left(\mathfrak{g}\right)$-modules $V$ and $W$ is defined
as $R_{V,W}:v\otimes w\mapsto R\left(w\otimes v\right)$.

\paragraph*{Action of a braid group on $U_{q}\left(\mathfrak{g}\right)$-rep
\cite[p. 276]{charipressley}}

If $V$ is a $U_{q}\left(\mathfrak{g}\right)$-module, action of the
universal $R$-matrix on $V^{\otimes3}$ satisfies the Yang-Baxter
equation
\[
\left(R\otimes1\right)\left(1\otimes R\right)\left(R\otimes1\right)=\left(1\otimes R\right)\left(R\otimes1\right)\left(1\otimes R\right)
\]
But this is also the braiding equation for the action of a braid
group $B_{n}$ on $V^{\otimes n}$. The $R$-matrix defines the representation
of the braid group $B_{n}$ that acts on $V^{\otimes n}$ as functorial
isomorphisms.

\section{Common framework}

In this section, we will lay out the common foundation for classical
and quantum shuffle algebras. Having this done, we will only need
to address the details specific to each case in the subsequent sections.\pagebreak{}

\subsection{Tensor space $T\left(V\right)$}
\begin{defn}[Tensor space $T\left(V\right)$]

For a context-specific vector space $V$ and $n\in\mathbb{N}$, we
define tensor space $T\left(V\right)$ as 
\[
T\left(V\right)=\oplus_{k=0}^{+\infty}V^{\otimes k}
\]
\end{defn}

In \cite{Radford}, the classical shuffle algebra is defined such
that it is $T\left(V\right)$ as a vector space. However, the bases
in terms of Lyndon words are constructed for its subspace generated
by $V$ \cite[p. 446]{Radford}, as well as the whole space $T\left(V\right)$
by adding an appropriately defined Lyndon word of zero length or the
multiplicative identity element of the underlying field of $V$.

A similar situation holds true quantum shuffle algebras. A quantum
shuffle algebra is an algebra and coalgebra structure defined on the
tensor space $T\left(V\right)$. The bases in terms of Lyndon words
are constructed for the subalgebra of $T\left(V\right)$ generated
by $V$, denoted as $S_{\sigma}\left(V\right)$ and called quantum
symmetric algebra or bialgebra (or Hopf algebra) of type one \cite[p. 407]{MRosso},
as well as for the whole tensor space $T\left(V\right)$.

This naturally follows from starting with and intending to use the
bases of $V$ to construct shuffle algebras.

\begin{defn}[Braiding on $T\left(V\right)$ \cite{jian,dechela2}]

Consider an arbitrary braiding $\tau\in End\left(V\otimes V\right)$
. Then $\tau$ satisfies the quantum Yang-Baxter braiding equation
(or, ``braiding equation'') on $V^{\otimes3}$ 
\[
\left(1\otimes\tau\right)\circ\left(\tau\otimes1\right)\circ\left(1\otimes\tau\right)=\left(\tau\otimes1\right)\circ\left(1\otimes\tau\right)\circ\left(\tau\otimes1\right)
\]

A vector space $V$ with braiding $\tau$ that satisfies these conditions
is called ``braided vector space''. The braiding $\tau$ induces
an action of the Artin braid group $B_{n}$ on $V^{\otimes n}\subset T\left(V\right)$
for all $n\in\mathbb{N}$.

The standard generators of $B_{n}$ are denoted as $\sigma_{1},\cdots,\sigma_{n-1}$.
For all $i\in\left[\left[1,n-1\right]\right]$, the generator $\sigma_{i}$
is defined as $1^{\otimes\left(i-1\right)}\otimes\tau\otimes1^{\otimes\left(n-i-1\right)}$
and can be written as $\left(i\mapsto i+1\right)$ \cite{braidgroups}.
Its inverse $\sigma_{i}^{-1}$ can be written as $\left(i+1\mapsto i\right)$.
The generator $\sigma_{i}$ acts on $V^{\otimes n}$ as follows:
\[
\sigma_{i}:\otimes_{j=1}^{n}V_{\left(j\right)}\mapsto\left(\otimes_{j=1}^{i-1}V_{\left(j\right)}\right)\otimes\tau\left(V_{\left(i\right)}\otimes V_{\left(i+1\right)}\right)\otimes\left(\otimes_{j=i+1}^{n}V_{\left(j\right)}\right)
\]
\end{defn}

\subsection{Symmetric and braid group notations}

We use the following symmetric and braid group notations \cite[p. 401]{MRosso}:
\begin{itemize}
\item $\sum_{n}$: the symmetric group of $\left\{ 1,2,\cdots,n\right\} $
\item $s_{i}$: transposition $\left(i,i+1\right)$
\item $B_{n}$: Artin braid group on $n$ strands
\item $\sigma_{i}$ (or $\left(i\mapsto i+1\right)$) for $i\in\left[\left[1,n-1\right]\right]$:
the $i$-th generator of $B_{n}$
\item $\left(l_{1},\cdots,l_{r}\right)$-shuffle for $l_{1}+\cdots+l_{r}=n$:
the set of permutations $w$ such that
\[
w\left(1\right)<w\left(2\right)<\cdots<w\left(l_{1}\right)
\]
\[
w\left(l_{1}+1\right)<w\left(l_{1}+2\right)<\cdots<w\left(l_{1}+l_{2}\right)
\]
\[
\cdots
\]
\[
w\left(l_{1}+\cdots+l_{r-1}+1\right)<\cdots<w\left(n\right)
\]
\item $\sum_{\left(l_{1},\cdots,l_{r}\right)}$ for $l_{1}+\cdots+l_{r}=n$:
the set of $\left(l_{1},\cdots,l_{r}\right)$-shuffles
\item $\sum_{l_{1},n-l_{1}}$ for $l_{1}\le n$: the set of $\left(l_{1},n-l_{1}\right)$-shuffles
\item $l\left(w\right)$: the length of the reduced expression of permutation
$w$ in terms of standard generators $s_{i}$
\item $T_{w}$ for $w\in\sum_{n}$: the lift of $w$ in $B_{n}$, also known
as ``Matsumoto section'' \cite{matsumoto}, and defined as $w=s_{i_{1}}\cdots s_{i_{l\left(w\right)}}\mapsto T_{w}=\sigma_{i_{1}}\cdots\sigma_{i_{l\left(w\right)}}$
\item $\mathscr{B}_{\left(l_{1},\cdots,l_{r}\right)}=\sum_{w\in\sum_{\left(l_{1},\cdots,l_{r}\right)}}T_{w}$
\item $\mathscr{\tilde{B}}_{\left(l_{1},\cdots,l_{r}\right)}=\sum_{w\in\sum_{\left(l_{1},\cdots,l_{r}\right)}}T_{w}^{-1}$
\end{itemize}

\subsection{Matsumoto section}

Symmetric group $\sum_{n}$ is generated by $n-1$ transpositions
$s_{j}=\left(j,j+1\right)$, where $j\in\left[\left[1,n-1\right]\right]$.
Likewise, braid group $B_{n}$ is generated by $n-1$ generators of
$\sigma_{j}=\left(j\mapsto j+1\right)$, where $j\in\left[\left[1,n-1\right]\right]$.
While verbs 'transpose' and 'permute' are defined for symmetric groups,
there is no universally recognized verb usage that fully expresses
respective braiding operations for braid groups. To remedy this, we
will extend this symmetric group terminology to braid groups. We will
say that the generators of $B_{n}$ transpose a pair of elements by
analogy with the symmetric group, with the caveat that $\sigma_{j}\ne\sigma_{j}^{-1}$
for all $j\in\left[\left[1,n-1\right]\right]$. We will say that elements
of $B_{n}$ permute elements of a tensor space (or a string) by the
same analogy. For any $s_{j}\in\sum_{n}$, we can define its lift
in $B_{n}$ (or, ``Matsumoto section'') as $T_{s_{j}}=\left\{ \sigma_{j},\sigma_{j}^{-1}\right\} $,
but only $T_{s_{j}}=\sigma_{j}$ applies for shuffle algebra purposes
by their construction. For a reduced expression on $w=s_{1}\cdot\cdots\cdot s_{l\left(w\right)}\in\sum_{n}$,
we define $T_{w}$ as $T_{w}=T_{s_{1}}\cdot\cdots\cdot T_{s_{l\left(w\right)}}$.

\subsection{Shuffle product on tensor space $T\left(V\right)$}
\begin{prop}[Shuffle product on $T\left(V\right)$ \cite{MRosso,Radford}]
\end{prop}

Assume that braid groups $B_{n},n\in\mathbb{N}$ act on $T\left(V\right)$.
Let $x_{1},\cdots,x_{n}$ be in $V$. We define an associative algebra
structure on $T\left(V\right)$, given by the shuffle product

\[
\left(x_{1}\otimes\cdots\otimes x_{p}\right)\cdot\left(x_{p+1}\otimes\cdots\otimes x_{n}\right)=\sum_{w\in\sum_{p,n-p}}T_{w}\left(x_{1}\otimes\cdots\otimes x_{n}\right)
\]

\subsection{Gradings of $T\left(V\right)$}

Let $\left(X,\le\right)$ be a totally ordered set. Let $S=\left(X\right)$
be the free semigroup generated by $X$.
\begin{defn}[Grading of $S$ \cite{Radford}]

Let $U=N^{\left(X\right)}$ be the additive semigroup of all functions
from $X$ to the natural numbers $N=\left\{ 0\right\} \cup\mathbb{N}$
which have finite support. For $u\in U$ let $S\left(u\right)\subseteq S$
bet the set of all $a=\prod_{j=1}^{k}x_{j}\in S$ such that $u\left(x\right)$
is the number of $x$'s in the factorization of $a$ for all $x\in X$.
We note that $S\left(u\right)$ is finite.
\end{defn}

\begin{defn}[$U$-graded coalgebra (bialgebra) \cite{Radford}]
\label{def:uG}

If $U$ is any commutative (resp. additive) semigroup, a coalgebra
(resp. bialgebra) $A$ is $U$-graded if for each $u\in U$ there
exists a subspace $A\left(u\right)$ such that (a)-(c) (resp. (a)-(d))
from the following are satisfied:

\begin{enumerate}
\item $A=\oplus_{u\in U}A\left(u\right)$, where $A\left(0\right)=k$;
\item $\epsilon\left(A\left(u\right)\right)=0$ if $u\ne0$;
\item $\Delta A\left(u\right)\subseteq\sum_{v+w=u}A\left(v\right)\otimes A\left(w\right)$;
\item $A\left(u\right)A\left(v\right)\subseteq A\left(u+w\right)$ for all
$u,v\in U$; and
\item $A\left(u\right)$ is finite-dimensional for all $u\in U$.
\end{enumerate}
An element $a\in A$ can uniquely decomposed as $a=\oplus a_{u}$,
where $a_{u}\in A\left(u\right)$ for all $u\in U$.
\end{defn}

$U$-grading was used in \cite{Radford} for classical shuffle algebras
but not in \cite{MRosso} for quantum shuffle algebras. While it may
not be essential for shuffle algebras, we find it to be useful for
understanding the structure of shuffle algebras and for their construction.
We decided to use $U$-grading for both classical and quantum shuffle
algebras in this work, and we found the application of $U$-grading
to quantum shuffle algebras to be an interesting exercise.

\begin{prop}[Gradings of $T\left(V\right)$ \cite{MRosso,Radford}]
\textcolor{white}{\scriptsize{}T}

\begin{enumerate}
\item $T\left(V\right)$ is naturally graded with $T^{k}V$ being the $k$-grade
subspace.
\item Let $T\left(V\right)\left(u\right)$ be the linear span of $v_{a}$'s
where $a\in S\left(u\right)$. This defines $U$-grading of $T\left(V\right)$.
\end{enumerate}
\end{prop}

These gradings are compatible with shuffle multiplication and with
universal construction of a shuffle algebra.

\subsection{Ordering and Lyndon words in $S$}
\begin{defn}[Total ordering of $S$ \cite{MRosso,Radford}]
\textcolor{white}{\scriptsize{}T}

\begin{enumerate}
\item A total ordering $\le$ on the set $S$ is defined by lexicographic
ordering, with the convention that $a\cdot b\le a$ for any $a,b\in S$.
\item We call an element $p\in S$ a prime (or, a ``Lyndon word'') if,
for any splitting $p=a\cdot b$ with $a,b\in S$, we have $b<p$.
We denote the set of primes in $S$ as $P$.
\end{enumerate}
\end{defn}

\begin{defn}[Prime factorization in $S$]

Let $a\in S$. We will call a factorization $\prod_{j=1}^{k}p_{j}=a$
a prime factorization of $a$ if each $p_{j}$ is a prime but cannot
itself be factorized into two or more primes.
\end{defn}

\begin{prop}[Existence of prime factorization in $S$]

Any $a\in S$ has a prime factorization.
\end{prop}

\begin{proof}
Let $a\in S$. Consider the factorization $\prod_{j=1}^{k}x_{j}=a$,
where $x_{j}\in X$ for all $j$. Each $x_{j}$ is a prime because
$X\subset P$ by definition of a prime. We can use the following iterative
method with $\prod_{j=1}^{k}x_{j}=a$ as the initially considered
factorization:

\begin{enumerate}
\item If the currently considered factorization $\prod_{j=1}^{k}p_{j}=a$
is a prime factorization, stop, as we have achieved the objective.
\item Otherwise, there must be a prime $p=\prod_{j=j_{1}}^{j_{2}}p_{j}$
that prevents $\prod_{j=1}^{k}p_{j}=a$ from satisfying the definition
of a prime factorization. We replace the currently considered factorization
of $a$ by $\prod_{j=1}^{j_{1}-1}p_{j}\cdot p\cdot\prod_{j=j_{2}+1}^{k}p_{j}$,
decreasing the length of factorization by at least one.
\item Go to step 1.
\end{enumerate}
\end{proof}
Since $k$ is finite, this iterative method completes in a finite
number of steps, yielding a prime factorization of $a$.

\begin{prop}[Unique prime factorization in $S$ \cite{MRosso,Radford}]

Prime factorization of any $a\in S$ is unique.
\end{prop}

\begin{proof}
Let $a\in S$ and suppose that $\prod_{j=1}^{k}p_{j}$ and $\prod_{j=1}^{r}p_{j}^{\prime}$
are prime factorizations of $a$. We need to prove that $k=r$ and
$p_{i}=p_{i}^{\prime}$ for all $i$.

Indeed:

\begin{enumerate}
\item The case $r=k=1$ is trivial.
\item Suppose that $r>1$ or $k>1$. We will prove that $p_{n}=p_{n}^{\prime}$
for $n=1$. Suppose $p_{1}\ne p_{1}^{\prime}$. Then either $p_{1}=\left(\prod_{j=1}^{m}p_{j}^{\prime}\right)\cdot s$
or $p_{1}^{\prime}=\left(\prod_{j=1}^{m}p_{j}\right)\cdot s$ for
some $m\in\mathbb{N}$ and $s\in S$ such that $s$ is shorter or
equal in length to $p_{m+1}^{\prime}$ or $p_{m+1}$ respectively.
Without the loss of generality, let $p_{1}=\left(\prod_{j=1}^{m}p_{j}^{\prime}\right)\cdot s$
for some $s\in S$. By definition of $s$, $s\ge p_{m+1}^{\prime}$.
Because $p_{1}$ is a prime, $p_{1}>s$ and therefore, $p_{1}^{\prime}>p_{1}>s\ge p_{m+1}^{\prime}$.
This contradicts $p_{1}^{\prime}\cdot\cdots\cdot p_{r}^{\prime}$
being a prime factorization. We can use the proof of Proposition \ref{prop:FormUPF}
here to justify this contradiction because it doesn't use the uniqueness
property of prime factorization.
\item Applying the method of mathematical induction on $n$ with item 2
of this proof as its base case and a trivial inductive step, we get
$p_{n}=p_{n}^{\prime}$ for all $n\le\max\left(\left\{ r,k\right\} \right)$.
\item Now suppose that $r\ne k$. Without the loss of generality, suppose
$r>k$. Then 
\[
a=\prod_{j=1}^{k}p_{j}=\prod_{j=1}^{k}p_{j}\cdot\prod_{j=k+1}^{r}p_{j}^{\prime}
\]
But by definition of total ordering $<$, 
\[
\prod_{j=1}^{k}p_{j}>\prod_{j=1}^{k}p_{j}\cdot\prod_{j=k+1}^{r}p_{j}^{\prime}
\]
Contradiction.
\end{enumerate}
We have proved that $r=k$ and $p_{i}=p_{i}^{\prime}$ for all $i$.
\end{proof}

\begin{prop}[Form of unique prime factorization \cite{MRosso,Radford}]
\label{prop:FormUPF}

For any $a\in S$, its unique prime factorization has the form $a=\prod_{j=1}^{k}p_{j}^{n_{j}}$,
where $p_{j}<p_{j+1}$ for all $1\le j\le k-1$ and $n_{j}\in\mathbb{N}$
for all $1\le j\le k$. 
\end{prop}

\begin{proof}
Indeed, suppose that $p_{j_{0}}\ge p_{j_{0}+1}$ for some $1\le j_{0}\le k-1$.
Then one of the following is true:

\begin{enumerate}
\item $p_{j_{0}}=p_{j_{0}+1}$. Then the prime factorization can be rewritten
as $a=\prod_{j=1}^{j_{0}-1}p_{j}^{n_{j}}\cdot p_{j_{0}}^{n_{j_{0}}+n_{j_{0}+1}}\cdot\prod_{j=j_{0}+2}^{k}p_{j}^{n_{j}}$,
where we omit any terms $p_{j}$ for $j\le0$ or $j\ge k+1$.
\item $p_{j_{0}}>p_{j_{0}+1}$. Then $a$ contains the term $p_{j_{0}}\cdot p_{j_{0}+1}$,
where $p_{j_{0}}>p_{j_{0}+1}$, but then by definition of prime, $p_{j_{0}}\cdot p_{j_{0}+1}$
itself is a prime, and there is a contradiction with $a=\prod_{j=1}^{k}p_{j}^{n_{j}}$
being a prime factorization.
\end{enumerate}
\end{proof}

\begin{thm}[Relation between $T_{w}\left(a\right)$ and $a$ \cite{MRosso,Radford}]
\label{thm:TwAA}

Suppose $a\in S$ has prime factorization $a=\prod_{i=1}^{s}p_{i}^{n_{i}}$.
Then

\begin{enumerate}
\item $T_{w}\left(a\right)\ge a$ for $w\in\sum_{\left(l_{1},\cdots,l_{s}\right)}$,
where $l_{i}$ is the length of $p_{i}$.
\item Let $w\in\sum_{\left(l_{1},\cdots,l_{s}\right)}$. Then $T_{w}\left(a\right)=a$
if and only if $w$ is in the subgroup $\sum_{n_{1}}\times\cdots\times\sum_{n_{s}}$
of ``block permutations'', permuting only $p_{i}$'s among themselves
for each $i$.
\end{enumerate}
\end{thm}

\begin{proof}
We note that for $a=p_{1}$, we have $T_{w}\left(a\right)=T_{w}\left(p_{1}\right)=p_{1}=a$
because $T_{\sum_{\left(l_{1}\right)}}=\left\{ Id\right\} $, so \emph{a
forteriori} $T_{w}\left(a\right)\ge a$. However, we will use the
fact that item (1) is true for $s=0$ vacuously and see that this
makes sense as we use it. Now suppose that $a=\prod_{i=1}^{s}p_{j}^{n_{i}}$
for $s=r\ge2$. Decompose $p_{i}=\prod_{j=1}^{l_{i}}x_{ij}$ for all
$1\le i\le s$. Also, we denote the decomposition of any $z\in S$
in $X$ as $z=\prod_{j}z_{j}$. By definition of prime and definition
of prime decomposition, we have $x_{s1}>x_{ji}$ for all $j$ and
$i$ such that $\left(i,j\right)\ne\left(1,s\right)$. Let $\lambda_{a}\left(x\right)$
be the set of indices $\left\{ i\right\} $ such that $a_{i}=x$.
Then by definition of a $\left(l_{1},\cdots,l_{s}\right)$-shuffle,
we then have that $\lambda_{a}\left(x_{s1}\right)\subseteq\left\{ 1,l_{1},\cdots,\sum_{t=1}^{s}n_{t}l_{t}-l_{s}+1\right\} $.
We note that the restriction of $w$ to $w^{-1}\left(\left[1,\min\left(\lambda_{T_{w}\left(a\right)}\left(x_{s1}\right)\right)-1\right]\right)$
is equivalent to the restriction of some $u\in\sum_{\left(l_{1},\cdots,l_{s}-1\right)}$
to the same set. There are two possibilities:

\begin{description}
\item [{(1a)}] $i_{0}:=\min\left(\lambda_{T_{w}\left(a\right)}\left(x_{s1}\right)\right)<\lambda_{a}\left(x_{s1}\right)$.
Supposing that item (1) of this Theorem is true for $s=r-1\ge1$,
we have that $\prod_{i=1}^{i_{0}-1}T_{w}\left(a\right)_{i}\ge\prod_{i=1}^{i_{0}-1}a_{i}$,
and because $x_{s1}>x_{ji}$ for all $j$ and $i$ such that $\left(i,j\right)\ne\left(1,s\right)$,
we have that $T_{w}\left(a\right)_{i_{0}}>a_{i_{0}}$. Therefore,
$T_{w}\left(a\right)>a$.
\item [{(1b)}] $i_{0}:=\min\left(\lambda_{T_{w}\left(a\right)}\left(x_{s1}\right)\right)\ge\lambda_{a}\left(x_{s1}\right)$.
Then $\lambda_{T_{w}\left(a\right)}\left(x_{s1}\right)=\lambda_{a}\left(x_{s1}\right)$.
Suppose $l=w^{-1}\left(\max\left(\lambda_{T_{w}\left(a\right)}\left(x_{s1}\right)\right)\right)$.
Applying the definition of a $\left(l_{1},\cdots,l_{s}\right)$-shuffle,
we have 
\[
w\left(l\right)<\cdots<w\left(l+l_{s}-1\right)
\]
and therefore for the interval $I=\left[\left[l,l+l_{s}-1\right]\right]$
we have $w\left(I\right)=\left[\left[\sum_{t=1}^{s}n_{t}l_{t}-l_{s}+1,\sum_{t=1}^{s}n_{t}l_{t}\right]\right]$.
Iteratively applying this reasoning, we have that $\prod_{i=\sum_{t=1}^{s-1}n_{t}l_{t}+1}^{\sum_{t=1}^{s}n_{t}l_{t}}w\left(a\right)_{i}=\prod_{i=\sum_{t=1}^{s-1}n_{t}l_{t}+1}^{\sum_{t=1}^{s}n_{t}l_{t}}a_{i}$,
i.e. $w\in\sum_{\left(l_{1},\cdots,l_{s-1}\right)}\times\sum_{n_{s}}$.
Now, supposing that item (1) of this Theorem is true for $s=r-1\ge0$,
we have that $\prod_{i=1}^{\sum_{t=1}^{s-1}n_{t}l_{t}}T_{w}\left(a\right)_{i}\ge\prod_{i=1}^{\sum_{t=1}^{s-1}n_{t}l_{t}}a_{i}$.
Therefore, $T_{w}\left(a\right)\ge a$.
\end{description}
From discussion in case 1b see that $T_{w}\left(a\right)=a$ if and
only if case 1b applies to the last block of $a=\prod_{i=1}^{r}p_{j}^{n_{i}}$
for all $1\le r\le s$, i.e. if and only if $w\in\sum_{n_{1}}\times\cdots\times\sum_{n_{s}}$.
\end{proof}

\section{Classical shuffle algebras}

Considering that classical shuffle algebras can be obtained as the
$q\rightarrow1$ limit of quantum shuffle algebras and that our objective
in this section is to summarize some facts about classical shuffle
algebras, we will be concise here and refer the reader to Quantum
Shuffle Algebras section (or to \cite{Radford}) for proofs.

\subsection{Definition and notes}
\begin{defn}[Classical shuffle algebra \cite{Radford}]

A classical shuffle algebra $Sh\left(V\right)$ of $V$ is a commutative
strictly graded pointed irreducible Hopf algebra with shuffle product:

for $x_{1},\cdots,x_{n}$ in $V$,
\[
\left(x_{1}\otimes\cdots\otimes x_{p}\right)\cdot\left(x_{p+1}\otimes\cdots\otimes x_{n}\right)=\sum_{w\in\sum_{p,n-p}}T_{w}\left(x_{1}\otimes\cdots\otimes x_{n}\right)
\]

As a vector space, $Sh\left(V\right)=T\left(V\right)$. 
\end{defn}

In classical shuffle algebras, the $T_{w}$'s are actually $\left(l_{1},\cdots,l_{r}\right)$-shuffles.
The lift $T_{w}$ of $\left(l_{1},\cdots,l_{r}\right)$-shuffle $w$
to the braid group represents an increase in the structure complexity
with the purpose of implementing a quantum deformation in quantum
shuffle algebras. Viewing the classical shuffle algebras as quantum
shuffle algebras in the limit $q\rightarrow1$, for consistent notation
between classical and quantum shuffle algebras, we may either (1)
by abuse of notation, formally set $T_{w}=w$; or (2) use the braiding
relation $\sigma\left(v\otimes w\right)=w\otimes v$ in the classical
shuffle algebras -- these two options are practically equivalent.
\begin{prop}[Embedding bialgebras into the classical shuffle algebra \cite{Radford}]
\textcolor{white}{\scriptsize{}T}

\begin{enumerate}
\item Let $A$ be a sub-bialgebra of $Sh\left(V\right)$ for some vector
space $V$ over a field $k$. Then $A$ is a commutative pointed irreducible
Hopf algebra. If $\mathrm{char\:}k=p>0$, then $x^{p}=0$ for $x\in A^{+}$.
\item Let $V=P\left(A\right)$ be the space of primitives of a commutative
pointed irreducible Hopf algebra $A$. If $\mathrm{char\:}k=0$ or
$\mathrm{char\:}k=p>0$ and $x^{p}=0$ for $x\in A^{+}$, then $A$
is isomorphic to a sub-Hopf algebra of $Sh\left(V\right)$.
\end{enumerate}
\end{prop}

\subsection{Bases in terms of Lyndon words}

Here we continue where we left off in the previous section.
\begin{lem}[The $\alpha_{aa}$ coefficient in $X_{a}$ (see Proposition \ref{prop:BasesCSA})]

Let $a\in S$, and let $p_{1}^{n_{1}}\cdots p_{s}^{n_{s}}=a$ be its
prime factorization $\left(p_{1}<\cdots<p_{s}\right)$. Then the number
of $\sigma$'s in $\mathscr{B}_{\left(l_{1},\cdots,l_{s}\right)}$
such that $\sigma\left(a\right)=a$ is $n_{1}!\cdots n_{s}!$.
\end{lem}

\begin{proof}
Trivial. The method of mathematical induction may be used. See Quantum
Shuffle Algebras II section or \cite[p. 445]{Radford} for details.
\end{proof}
\begin{thm}[Bases of classical shuffle algebras in terms of Lyndon words \cite{Radford}]
\label{prop:BasesCSA}

Assume $\mathrm{char\:}k=0$. Let $a\in S\left(u\right)$ and $a=\prod_{i=1}^{s}p_{i}^{n_{i}}$
be its unique prime factorization. We define $X_{a}=\prod_{i=1}^{s}v_{p_{i}}^{n_{i}}$,
where quantum multiplication is used between the terms of the form
$v_{p_{i}}$. Then:

\begin{enumerate}
\item $X_{a},a\in S\left(u\right)$ form a basis of $T\left(V\right)\left(u\right)$;
and
\item the change of basis with respect to $a$ is triangular, i.e. there
exist $\alpha_{ab}\in k$ such that $X_{a}=\sum_{a\le b}\alpha_{ab}v_{b}$;
\item setting for $1\le i\le s$ $p_{i}=\prod_{j=1}^{l_{i}}x_{ij}$, we
have $\alpha_{aa}=\prod_{j=1}^{s}\left(n_{j}!\right)\ne0$;
\item the $X_{a}$'s, $a\in S$, form a linear basis of $T\left(V\right)$;
\item the $v_{p}$'s, $p\in P$, form a polynomial basis for $T\left(V\right)$;
\item $X_{a}\cdot X_{b}=X_{a\cdot b}$ for $a,b\in S$.
\end{enumerate}
\end{thm}

\begin{proof}
See the Proposition \ref{prop:BasesQSA} or \cite[p. 446]{Radford}.
\end{proof}
A similar theorem has been formulated in \cite[p. 447]{Radford} for
the case $\mathrm{char\:}k=p>0$.

Point 5 in Theorem \ref{prop:BasesCSA} elucidates the reason why
we call this construction of bases in terms of Lyndon words. While
any $v_{a},a\in S$ has a unique prime factorization of $a$ in terms
of Lyndon words, the resulting $v_{p}$ components compose $v_{a}$
using tensor multiplication, but for a polynomial basis, the operation
that should be used is multiplication. This is satisfied with shuffle
multiplication in Theorem \ref{prop:BasesCSA}.

\section{Quantum shuffle algebras}

We will follow the fundamental work of Rosso \cite{MRosso} towards
the definition of a quantum shuffle algebra. Afterwards, we will provide
formulations and our proofs for construction of bases of quantum shuffle
algebras in terms of Lyndon words.

\subsection{Hopf bimodules}

We will first define Hopf bimodules and provide an overview of their
relevant structure.
\begin{defn}[{Hopf bimodule \cite[p. 401]{MRosso}}]

Let $H$ be a $k$-Hopf algebra. A Hopf bimodule over $H$ is a $k$-vector
space $M$ given with an $H$-bimodule structure, a $H$-bicomodule
structure (i.e. left and right coactions $\delta_{L}:M\rightarrow H\otimes M$,
$\delta_{R}:M\rightarrow M\otimes H$ which commute in the following
sense: $\left(\delta_{L}\otimes Id\right)\delta_{R}=\left(Id\otimes\delta_{R}\right)\delta_{L}$,
and such that $\delta_{L}$ and $\delta_{R}$ are morphisms of $H$-bimodules.
\end{defn}

Taking tensor products over Hopf algebra $H$, Hopf bimodules form
a tensor category $\mathscr{E}$ \cite{MRosso,nichols}.
\begin{defn}[{Left and coinvariants of a Hopf bimodule \cite[p. 402]{MRosso}}]
\textcolor{white}{\scriptsize{}T}

\begin{enumerate}
\item The left coinvariant subspace $M^{L}$ of $M$ is defined as 
\[
M^{L}=\left\{ m\in M|\delta_{L}\left(m\right)=1\otimes m\right\} 
\]
It is a sub-right comodule of $M$ and inherits a structure of right
$H$-module by
\[
m\cdot h=\sum S\left(h_{\left(1\right)}\right)mh_{\left(2\right)}
\]
where $m\in M$ and $h\in H$.
\item The right coinvariant subspace $M^{R}$ of $M$ is defined as 
\[
M^{R}=\left\{ m\in M|\delta_{R}\left(m\right)=m\otimes1\right\} 
\]
It is a sub-left comodule of $M$ and inherits a structure of left
$H$-module by
\[
h\cdot m=\sum h_{\left(1\right)}mS\left(h_{\left(2\right)}\right)
\]
where $m\in M$ and $h\in H$.
\end{enumerate}
\end{defn}

\begin{prop}[{Properties of the right coinvariant \cite[p. 402]{MRosso}}]
\textcolor{white}{\scriptsize{}T}

\begin{enumerate}
\item The right coinvariant $M^{R}$ of a Hopf bimodule $M$ over $H$ is
a crossed module over $H$ in the sense of Yetter.
\item If $H$ and $M$ are finite-dimensional, it is a module over the quantum
double.
\item A morphism of Hopf bimodules induces on the space of right coinvariants
a morphism of crossed modules.
\end{enumerate}
\end{prop}

\subsection{Braidings in the category of Hopf bimodules}

Here we establish that the right coinvariant space $M^{R}$ of a Hopf
bimodule $M$ has a naturally defined braiding $\sigma\left(v\otimes w\right)\mapsto w\otimes v$,
where $v,w\in M^{R}$. This braiding is a foundation for construction
of a quantum shuffle algebra.
\begin{prop}[{Braiding in the category of Hopf bimodules \cite[p. 403]{MRosso}}]

Let $M$ and $N$ be $H$-Hopf bimodules.

\begin{enumerate}
\item There exists a unique morphism of $H$-bimodules $\sigma_{M,N}:M\otimes N\rightarrow N\otimes M$
such that, for $\omega\in M^{L}$ and $\eta\in M^{R}$ we have $\omega_{M,N}\left(\omega\otimes\eta\right)=\eta\otimes\omega$.
\item Furthermore, $\omega_{M,N}$ is an invertible morphism of bicomodules
and satisfies the following braid equation (where $M$, $N$, and
$P$ are Hopf bimodules):
\begin{multline*}
\left(I_{P}\otimes\omega_{M,N}\right)\left(\omega_{M,P}\otimes I_{N}\right)\left(I_{M}\otimes\omega_{N,P}\right)=\\
=\left(\omega_{N,P}\otimes I_{M}\right)\left(I_{N}\otimes\omega_{M,P}\right)\left(\omega_{M,N}\otimes I_{P}\right)
\end{multline*}
\end{enumerate}
\end{prop}

\begin{enumerate}
\item This makes $\mathscr{E}$ a braided tensor category \cite{MRosso,kassel,montgomery}.
\item $\sigma_{M,M}$ sends $M^{R}\otimes M^{R}$ into itself, acting as
\[
\sigma\left(x\otimes y\right)=\delta_{L}\left(x\right)\left(y\otimes1\right)
\]
and defines a representation $T$ of the braid group $B_{n}$ in $\left(M^{R}\right)^{\otimes n}$.
\end{enumerate}
\begin{prop}[{Equivalence between $\mathscr{E}$ and the category of crossed modules
\cite[p. 403]{MRosso}}]

The functor sending $M$ to $M^{R}$ is an equivalence of braided
tensor categories between $\mathscr{E}$ and the category of crossed
modules.
\end{prop}

\subsection{The cotensor Hopf algebra}

We recall that a classical shuffle algebra is, on one level, a commutative
strictly graded pointed irreducible Hopf algebra, and, on another
level, a tensor space $T\left(V\right)$ with shuffle multiplication.
Construction of a quantum shuffle algebra features a similar duality
(in the philosophical sense of the word). The technical part with
the tensor space $T\left(V\right)$, where $V=M^{R}$, is endowed
with conceptual meaning as quantum shuffle algebra using cotensor
Hopf algebra. A nontrivial isomorphism between the cotensor Hopf algebra
and $T\left(V\right)$ is a key link in this construction.
\begin{defn}[{Cotensor product and cotensor coalgebra \cite[p. 403]{MRosso}}]
\textcolor{white}{\scriptsize{}T}

\begin{enumerate}
\item Let $M$ and $N$ be $H$-Hopf bimodules. Their cotensor coproduct
$M\sqcup N$ is the kernel of 
\[
\delta_{R}\otimes I_{N}-I_{M}\otimes\delta_{L}:M\otimes N\rightarrow M\otimes H\otimes N
\]
\item The cotensor coalgebra is constructed on $M$ is 
\[
T_{H}^{c}\left(M\right)=H\oplus\oplus_{n\ge1}M^{\sqcup n}
\]
\end{enumerate}
\end{defn}

The following applies to the structure of tensor algebra 
\[
T_{H}\left(M\right)=H\oplus\oplus_{n\ge1}M^{\otimes_{H}n}
\]

\begin{enumerate}
\item If $A$ is an $H$-bimodule, $h\in H$, and $a$ and $b\in A$, then
\[
h\left(ab\right)=\sum h_{\left(1\right)}\left(a\right)h_{\left(2\right)}\left(b\right)
\]
\item If $f$ and $g:M\rightarrow A$ are two $H$-bimodule maps, the map
$f\cdot g:M\otimes_{H}M\rightarrow A$ is defined for $m\in M$ and
$n\in M$ as
\[
f\cdot g:m\otimes_{H}n\mapsto f\left(m\right)f\left(n\right)
\]
\end{enumerate}
\begin{prop}[{Power of the sum of two bimodule maps \cite[p. 404]{MRosso}}]

Let $A$ be an $H$-bimodule algebra and $f$ and $g:M\rightarrow A$
two bimodule maps. Assume that $g\cdot f=\left(f\cdot g\right)\circ\sigma$.
Then the map $\left(f+g\right)^{n}:M^{\otimes_{H}n}\rightarrow A$
is given by
\[
\left(f+g\right)^{n}=\sum_{k=0}^{n}\left(f^{k}\cdot g^{n-k}\right)\circ\tilde{\mathscr{B}}_{k,n-k}
\]
\end{prop}

\begin{lem}[{Relation between coactions and braiding \cite[p. 404]{MRosso}}]

\[
\delta_{l}\cdot\delta_{R}=\left(\delta_{R}\cdot\delta_{R}\right)\circ\sigma
\]
\end{lem}

\begin{prop}[{Coproduct on $T\left(M\right)$ \cite[p. 404]{MRosso}}]
\label{prop:Coproduct}

The coproduct on $T_{H}\left(M\right)$ is given by: for $\left(x_{1},\cdots,x_{n}\right)\in M^{\otimes_{H}n}$,
\[
\triangle\left(x_{1},\cdots,x_{n}\right)=\sum_{k=0}^{n}\left(\delta_{R}^{k}\cdot\delta_{L}^{n-k}\right)\circ\tilde{\mathscr{B}}_{k,n-k}
\]
\end{prop}

\begin{defn}[{Universal property of $T_{H}^{c}\left(M\right)$ \cite[p. 405]{MRosso}}]

The product on $T_{H}^{c}\left(M\right)$ is a unique coalgebra map
\[
T_{H}^{c}\left(M\right)\otimes T_{H}^{c}\left(M\right)\rightarrow T_{H}^{c}\left(M\right)
\]
induced by the product on $H$ and by left or right coaction on $H\otimes M+M\otimes H$.
\end{defn}

We denote $V=M^{R}$ for simplicity.
\begin{prop}[{Shuffle multiplication and structure of $T\left(V\right)$ \cite[p. 405]{MRosso}}]
\textcolor{white}{\scriptsize{}T}

\begin{enumerate}
\item There is an associative algebra structure on $T\left(V\right)$, given
by: for $x_{1},\cdots,x_{n}$ in $V$,
\[
\left(x_{1}\otimes\cdots\otimes x_{p}\right)\cdot\left(x_{p+1}\otimes\cdots\otimes x_{n}\right)=\sum_{w\in\sum_{p,n-p}}T_{w}\left(x_{1}\otimes\cdots\otimes x_{n}\right)
\]
We call this ``shuffle multiplication'' or, more formally, ``quantum
shuffle multiplication''. In other publications, it is also called
``quantum shuffle product'' and ``braided shuffle product'' \cite[p. 107]{takeuchi}.
\item The diagonal coaction on $H$ on each $V^{\otimes n}$ gives $T\left(V\right)$
an $H$-comodule structure $\delta_{L}:T\left(V\right)\rightarrow H\otimes T\left(V\right)$,
and $\delta_{L}$ is an algebra homomorphism.
\item For the diagonal action of $H$ on each $V^{\otimes n}$, $T\left(V\right)$
is an $H$-module algebra, and $T\left(V\right)\otimes H$ inherits
the crossed product algebra structure.
\item The following defines a coalgebra structure on $T\left(V\right)\otimes H$:
for $\left(v_{1},\cdots,v_{n}\right)$ in $V$ and $h$ in $H$,
\begin{multline*}
\triangle\left(\left(v_{1},\cdots,v_{n}\right)\otimes h\right)=\\
=\sum_{k=0}^{n}\left[\left(v_{1},\cdots,v_{k}\right)\otimes v_{k+1\left(-1\right)}\cdots v_{n\left(-1\right)}h_{\left(1\right)}\right]\\
\otimes\left[\left(v_{k+1\left(0\right)},\cdots,v_{n\left(0\right)}\right)\otimes h_{\left(2\right)}\right]
\end{multline*}
\item The algebra structure of 3 and coalgebra structure of 4 are compatible
and make $T\left(V\right)\otimes H$ a Hopf algebra.
\end{enumerate}
\end{prop}

\begin{prop}[{Embedding of $T\left(V\right)$ in $T_{H}^{c}\left(M\right)$ \cite[p. 406]{MRosso}}]
\textcolor{white}{\scriptsize{}T}

\begin{enumerate}
\item There is a natural embedding of $\phi$ of $T\left(V\right)$ in $T_{H}^{c}\left(M\right)$
given on homogenous elements of degree $n$ by:
\begin{multline*}
\phi\left(v^{1},\cdots,v^{n}\right)=\\
=\sum v^{1}v_{\left(-1\right)}^{2}\cdots v_{\left(-n+1\right)}^{n}\otimes v_{\left(0\right)}^{2}v_{\left(-1\right)}^{3}\cdots v_{\left(-n+2\right)}^{n}\otimes\cdots\otimes v_{\left(0\right)}^{n}
\end{multline*}
whose image is the subspace of right coinvariants.
\item There is an isomorphism of right module and comodule $\tilde{\phi}:T\left(V\right)\otimes H\rightarrow T_{H}^{c}\left(M\right)$
\begin{multline*}
\tilde{\phi}\left[\left(v^{1},\cdots,v^{n}\right)\otimes h\right]=\\
=\sum v^{1}v_{\left(-1\right)}^{2}\cdots v_{\left(-n+1\right)}^{n}h_{\left(1\right)}\\
\otimes v_{\left(0\right)}^{2}v_{\left(-1\right)}^{3}\cdots v_{\left(-n+2\right)}^{n}h_{\left(2\right)}\otimes\cdots\otimes v_{\left(0\right)}^{n}h_{\left(n\right)}
\end{multline*}
\item The subspace of right coinvariants is a subalgebra of $T_{H}^{c}\left(M\right)$.
\end{enumerate}
\end{prop}

\begin{thm}
\label{thm:TheMapQSA}The map $\tilde{\phi}$ is a Hopf algebra isomorphism.
\cite[p. 405]{MRosso}
\end{thm}

Below, we will show bases construction in terms of Lyndon words for
$T\left(V\right)$.

\subsection{Quantum symmetric algebra}

The term ``quantum symmetric algebra'' was introduced by Rosso \cite{MRosso}
and refers in part to quantization $\sigma\left(v\otimes w\right)=q_{vw}w\otimes v$
of the action $\sigma\left(v\otimes w\right)=w\otimes v$ of a symmetric
group on a tensor space.
\begin{defn}[{Quantum symmetric algebra \cite[p. 407]{MRosso}}]

The sub-Hopf algebra $S_{H}\left(M\right)$ of $T_{H}^{c}\left(M\right)$
generated by $M$ and $H$ is a Hopf bimobule, and its subspace of
right coinvariants is isomorphic, via $\phi$, to the subalgebra of
$T\left(V\right)$ generated by $V$. We will call this a ``quantum
symmetric algebra'' and denote it by $S_{\sigma}\left(V\right)$.
\end{defn}

Quantum symmetric algebra is also known as ``Nichols algebra'' by
N. Andruskiewitsch and H.-J. Schneider \cite{anschneider} and ``bialgebras
of type one'' by W. Nichols \cite{nichols}.

In general, braided shuffle algebras naturally satisfy the quantum
Serre relations \cite[p. 108]{takeuchi}.

The braiding on $S_{\sigma}\left(V\right)$ is induced by the diagonal
braiding in $V\times V$ given by $\sigma\left(e_{i}\otimes e_{j}\right)=q_{ij}\left(e_{j}\otimes e_{i}\right)$
and is encoded in the $N\times N$ matrix $\left(q_{ij}\right)$.

Since the braiding on $S_{\sigma}\left(V\right)$ is diagonal, we
can extend the method of construction of bases in terms of Lyndon
words for classic shuffle algebras to $S_{\sigma}\left(V\right)$
almost verbatim.

\subsection{Universal construction in braid category}

Assume that we have used the braid group generator $\sigma$ to construct
a representation of the braid group $B_{n}$ in $V^{\otimes n}$.
If we define a graded multiplication $sh$ as in Proposition \ref{prop:AssocAlgStuc},
or, more specifically, a shuffle multiplication, on the tensor space
$T\left(V\right)$ generated by $V$, we may wish to ensure that algebra
and coalgebra structures on $T\left(V\right)$ are compatible. To
that end, we can (1) use the coproduct definition from Proposition
\ref{prop:Coproduct} with trivial coactions, and (2) change the product
in $T\left(V\right)\otimes T\left(V\right)$.
\begin{prop}[{Associative algebra structure on $T\left(V\right)\otimes T\left(V\right)$
\cite[p. 407-408]{MRosso}}]
\label{prop:AssocAlgStuc}

Let $\left(V,\sigma\right)$ be a braided space.
\begin{enumerate}
\item The following defines an associative algebra structure on $T\left(V\right)\otimes T\left(V\right)$:
for $p$ and $q$ two positive integers, let $w_{p,q}$ be the permutation:
\[
\left(\begin{array}{ccccccc}
1 & 2 & \cdots & q & q+1 & \cdots & p+q\\
p+1 & p+2 & \cdots & p+q & 1 & \cdots & p
\end{array}\right)
\]
and let $T_{w_{p,q}}$ be the associative element in the braid group
$B_{p+q}$: it acts in $V^{\otimes p+q}$ and can be seen as a ``generalized
flip'' from $V^{\otimes p}\otimes V^{\otimes q}$ to $V^{\otimes q}\otimes V^{\otimes p}$.
Then the product sends $\left(V^{\otimes n}\otimes V^{\otimes p}\right)\otimes\left(V^{\otimes q}\otimes V^{\otimes m}\right)$
to $V^{\otimes n+q}\otimes V^{\otimes p+m}$ and it is the composition:
$\left(sh\otimes sh\right)\circ\left(Id\otimes T_{w_{p,q}}\otimes Id\right)$
where $sh$ denotes the product on $T\left(V\right)$ defined above.
\item Then $\Delta:T\left(V\right)\rightarrow T\left(V\right)\otimes T\left(V\right)$
is an algebra homomorphism.
\end{enumerate}
\end{prop}

\begin{prop}[{Shuffle algebra \cite[p. 408]{MRosso}}]
\textcolor{white}{\scriptsize{}T}

\begin{enumerate}
\item The shuffle algebra $\mathscr{S}$ in $\mathscr{B}$ is, as an object,
the direct sum of all objects $n$.
\item The product is given by the direct sum of its ``homogenous components'':
for all $n$ and $m$ in $\mathbb{N}$, one has a component in 
\[
Mor\left(n\otimes m,n+m\right)
\]
 which is the sum $T_{w}\in B_{n+m}$, $w$ ranging in the set of
$\left(n,m\right)$-shuffles of $\sum_{n+m}$.
\item The coalgebra structure is also given by the direct sum of its ``homogenous
components'': for all $n$, $p$, $q$ in $\mathbb{N}$ such that
$p+q=n$, the component in $Mor\left(n,p\otimes q\right)$ is the
canonical morphism $n\rightarrow p\otimes q$.
\item These product and coproduct are compatible if we give $\mathscr{S}\otimes\mathscr{S}$
the algebra structure deduced from the one on $\mathscr{S}$ by first
twisting the braiding:
\[
\left(Id\otimes T_{w_{p,q}}\otimes Id\right):\left(n+p\right)\otimes\left(q+m\right)\rightarrow\left(n+q\right)\otimes\left(p+m\right)
\]
\end{enumerate}
\end{prop}

\subsection{Block permutation of Lyndon words}

In classical shuffle algebras, the number of all possible block permutation
of $p$'s in $p^{n}\in S$ is trivially computed as $n!$. In case
of diagonal braiding $\sigma\left(v\otimes w\right)=q_{vw}w\otimes v$,
this is a bit more complicated. We address this with a lemma and a
corollary on block permutation on Lyndon words.
\begin{lem}[Block permutation of two equal primes]
\label{lem:BlockPerm2Primes}

Let $a\in S$ and $a=p^{2}$ its prime factorization. We denote the
string decomposition of $p$ as $\prod_{j=1}^{L}e_{j}$, where $L$
is its length. Let $w$ be in $\sum_{2}$, a subgroup of ``block
permutations'' that permutes the $p$'s among themselves. Whenever
we speak of ``block permutations'', we will express $w$ and $T_{w}$
in terms of permutations and braidings that are applied to the ``blocks''.
Let the braiding in $V\otimes V$ be given by $\sigma\left(e_{i}\otimes e_{j}\right)=q_{ij}\left(e_{j}\otimes e_{i}\right)$,
and let $Q=\prod_{k,l\in\left[\left[1,L\right]\right]}q_{kl}$. Then
either

\begin{enumerate}
\item $w=Id$ and $T_{w}\left(v_{a}\right)=v_{a}$;
\item $w=s_{1}$, $T_{w}=\sigma_{1}=\left(1\mapsto2\right)$, and $T_{w}\left(v_{a}\right)=Qv_{a}$;
or
\item $w=s_{1}$, $T_{w}=\sigma_{1}^{-1}=\left(2\mapsto1\right)$, and $T_{w}\left(v_{a}\right)=Q^{-1}v_{a}$
(N/A).
\end{enumerate}
\end{lem}

\begin{proof}
(1) Trivial. (2) We will first consider the case when the rightmost
$p$ is transposed to the left. In that case, $T_{w}$ permutes each
of the elements $v_{x_{j}}$ composing the $p$'s such that they interchange
position. To do this, it is sufficient to permute each $v_{x_{j}}$
from the rightmost $p$ sequentially toward their respective place
in the beginning of the construction, starting from the leftmost one.
Each $v_{x_{j}}$ of the rightmost $p$ is thus transposed with one
$v_{x_{k}}$ for all $k\in\left[\left[1,L\right]\right]$ in the process.
Thus, permutation of each $v_{x_{j}}$ multiplies the construction
by $\prod_{k\in\left[\left[1,L\right]\right]}q_{x_{j}x_{k}}$. Overall,
for permutation of the $p$'s we have $T_{w}\left(X_{a}\right)=QX_{a}$.
(3) In case the rightmost $p$ is transposed to the right, $T_{w}$
is, as a whole and component-wise, an inverse of the the one in the
first case, so $T_{w}\left(X_{a}\right)=Q^{-1}X_{a}$.
\end{proof}
Only cases 1 and 2 in Lemma \ref{lem:BlockPerm2Primes} are applicable
to shuffle multiplication or construction of bases in terms of Lyndon
words by definition of shuffle multiplication (case 3 is not applicable).
\begin{cor}[Block permutation of $n$ equal primes]

Let $a\in S$ and $a=p^{n}$ its prime factorization. Let $w$ be
in $\sum_{n}$, a subgroup of ``block permutations'' that permutes
the $p$'s among themselves. (Whenever we speak of ``block permutations'',
we express $w$ and $T_{w}$ in terms of permutations and braidings
that are applied to ``blocks''.) Let the braiding in $V\otimes V$
be given by $\sigma\left(e_{i}\otimes e_{j}\right)=q_{ij}\left(e_{j}\otimes e_{i}\right)$,
and let $Q=\prod_{k,l\in\left[\left[1,L\right]\right]}q_{kl}$. Let
$T_{w}$ be the permutation of $p$'s that is a lift of $\left(j_{2},j_{1},j_{1}+1,\cdots,j_{2}-1\right)$,
where $j_{1}$ and $j_{2}$ correspond to $p$'s in positions $j_{1}$
and $j_{2}$ respectively such that $j_{1}<j_{2}$. Then, considering
the Lemma \ref{lem:BlockPerm2Primes} and that its item 3 is not applicable,
we have
\[
T_{w}\left(v_{a}\right)=Q^{\left(j_{2}-j_{1}\right)}v_{a}
\]
\end{cor}

\begin{proof}
Trivial, using the method of mathematical induction.
\end{proof}

\subsection{Bases in terms of Lyndon words}

For the case of quantum shuffle algebras, we need to only make a small
adjustment to Proposition \ref{prop:BasesCSA}, replacing factorials
by Mahonian numbers \cite{bona}.
\begin{prop}[Bases of quantum shuffle algebras in terms of Lyndon words]
\label{prop:BasesQSA}

Let $a\in S\left(u\right)$ and $a=\prod_{i=1}^{s}p_{i}^{n_{i}}$
be its unique prime factorization. We define $X_{a}=\prod_{i=1}^{s}v_{p_{i}}^{n_{i}}$,
where quantum multiplication is used between the terms of the form
$v_{p_{i}}$. Then:

\begin{enumerate}
\item $X_{a},a\in S\left(u\right)$ form a basis of $T\left(V\right)\left(u\right)$;
and
\item the change of basis with respect to $a$ is triangular, i.e. there
exist $\alpha_{ab}\in k$ such that $X_{a}=\sum_{a\le b}\alpha_{ab}v_{b}$;
\item setting for $1\le i\le s$ $p_{i}=\prod_{j=1}^{l_{i}}x_{ij}$ and
$Q_{i}=\prod_{k,l\in\left[\left[1,l_{i}\right]\right]}q_{x_{k}x_{l}}$,
we have $\alpha_{aa}=\prod_{j=1}^{s}\left(\left[n_{j}\right]_{Q_{j}}!\right)\ne0$
for all $Q_{i}\in k$ being different from unity, or $Q_{i}$'s indeterminate;
\item the $X_{a}$'s, $a\in S$, form a linear basis of $T\left(V\right)$;
\item the $v_{p}$'s, $p\in P$, form a polynomial basis for $T\left(V\right)$;
\item $X_{a}\cdot X_{b}=X_{a\cdot b}$ for $a,b\in S$.
\end{enumerate}
\end{prop}

\begin{proof}
Let $l_{i}$ be the length of $p_{i}$. By definition of shuffle multiplication,
we have $X_{a}=\mathscr{B}_{\left(l_{1},\cdots,l_{s}\right)}\left(v_{a}\right)$,
where $v_{a}=\otimes_{i=1}^{s}v_{p_{i}}^{n_{i}}$.

First we will prove (2). 

Indeed, we have $X_{a}=\mathscr{B}_{\left(l_{1},\cdots,l_{s}\right)}\left(v_{a}\right)=\sum_{w\in\sum_{\left(l_{1},\cdots,l_{s}\right)}}T_{w}\left(v_{a}\right)$.
By Lemma 1, $T_{w}\left(a\right)\ge a$ for all $w\in\sum_{\left(l_{1},\cdots,l_{s}\right)}$,
so $X_{a}=\sum_{b\ge a}\alpha_{ab}v_{b}$, where $\alpha_{ab}=0$
if $b\notin S\left(u\right)$.

Now we will prove (3).

We note that braiding in $V\otimes V$ is given by $\sigma\left(e_{i}\otimes e_{j}\right)=q_{ij}\left(e_{i}\otimes e_{j}\right)$.

We will rewrite $\left[n_{j}\right]_{Q_{j}}!$ as

\[
\left[n_{j}\right]_{Q_{j}}!=\prod_{i=1}^{n_{j}}\left[i\right]_{Q_{j}}=\prod_{i=1}^{n_{j}}\frac{1-Q_{j}^{i}}{1-Q_{j}^{1}}=\prod_{i=1}^{n_{j}}\sum_{m=0}^{i-1}Q_{j}^{m}
\]

We note that the expanded form of the Mahonian number formula 
\[
\left[n_{j}\right]_{Q_{j}}!=\prod_{i=1}^{n_{j}}\sum_{m=0}^{i-1}Q_{j}^{m}
\]
is valid even if $Q_{j}=1$.

From Theorem \ref{thm:TwAA}, we know that if $w\in\sum_{\left(l_{1},\cdots,l_{s}\right)}$,
then $T_{w}\left(a\right)=a$ if and only if $w$ is in the subgroup
$\sum_{n_{1}}\times\cdots\times\sum_{n_{s}}$ of ``block permutations'',
permuting only $p_{i}$'s among themselves for each $i$.

Therefore, it is sufficient to prove (3) for the case $s=1$, i.e.
$a=p_{1}^{n_{1}}$, with the decomposition $p_{1}=\prod_{j=1}^{l_{1}}x_{1j}$.

We will use the method of mathematical induction.

Step 1. We will prove the formula
\[
\alpha_{aa}=\left[n_{1}\right]_{Q_{1}}!
\]
 is valid for the case $n_{1}=1$. Indeed, then $\sum_{n_{1}}=\sum_{_{1}}$
consists only of the identity element and

\[
\alpha_{aa}=1=\left[n_{1}\right]_{Q_{1}}!
\]
as in that case $\left[n_{1}\right]_{Q_{1}}!=\prod_{i=1}^{1}\sum_{m\in\left\{ 0\right\} }Q_{1}^{m}=1$.

Step 2. We will prove that if the formula 
\[
\alpha_{aa}=\left[n_{1}\right]_{Q_{1}}!
\]
is valid for $n_{1}=r\in\mathbb{N}$, it is valid for $n_{1}=r+1$.
We note the natural inclusion $\iota:B_{r}\rightarrow B_{r+1}$. For
the permutations $T_{w_{1}}$ of the first $r$ $p_{1}$'s among themselves
\[
\left[r\right]_{Q_{1}}!=\prod_{i=1}^{r}\sum_{m=0}^{i-1}Q_{1}^{m}
\]
In case of $n_{1}=r+1$, to obtain the permutations $T_{w}$ of all
$p$'s among themselves, we take compositions $T_{w}=T_{w_{2}}\circ T_{w_{1}}$of
permutations $T_{w_{1}}$ of the first $r$ $p_{1}$'s among themselves
and the additional permutations $T_{w_{2}}$ of the last $p_{i}$
to each possible position relative to the first $r$ $p_{i}$'s. By
Lemma \ref{lem:BlockPerm2Primes}, if $T_{w}$ permutes the $p$ in
position $r+1$ to position $1\le i\le r$, then $T_{w}\left(X_{a}\right)=Q^{2t}X_{a}$.
We note that $T_{w}$ is then already written in terms of reduced
expressions by construction. Adding these factors and multiplying
by $\left[r\right]_{Q_{1}}!$, we get 

\[
T_{w}\left(X_{a}\right)=\left[r\right]_{Q_{1}}!\sum_{m=0}^{r}Q_{1}^{m}X_{a}=\left[r+1\right]_{Q_{1}}!X_{a}
\]
and
\[
\alpha_{aa}=\left[n_{1}\right]_{Q_{1}}!
\]
is valid for $n_{1}=r+1$.

By definition of a Mahonian number, assuming that $Q_{1}\ne1$, 
\[
\alpha_{aa}=\left[n_{1}\right]_{Q_{1}}!\ne0
\]

Now we will prove (1).

Consider the basis $\left\{ v_{a}|a=\prod_{j=1}^{k}x_{j}\in S\left(u\right),k\in\mathbb{N}\right\} $of
$S\left(u\right)$. Since $\alpha_{aa}\ne0$, the matrix $\left(\alpha_{ab}\right)$
has a nonzero determinant, and therefore $X_{a},a\in S\left(u\right)$
form a basis of $T\left(V\right)\left(u\right)$.

Item (4) follows from Definition \ref{def:uG} with exclusion of $0\in U$
(as it does not correspond to an element of $V$), i.e. 
\[
\oplus_{k=1}^{+\infty}V^{\otimes k}=\oplus_{0\ne u\in U}T\left(V\right)\left(u\right)
\]

Item (5) follows from (4) and from the fact that $X_{a}=\prod_{i=1}^{s}v_{p_{i}}^{n_{i}}$
can be expressed as a monomial in $v_{p}$'s. Item (6) follows by
definition of $X_{a}$: indeed, for $X_{a}=\prod_{i=1}^{s}v_{p_{i}}^{n_{i}}$
and $X_{b}=\prod_{i=1}^{s^{\prime}}v_{p_{i}^{\prime}}^{n_{i}^{\prime}}$
we have $X_{a}\cdot X_{b}=\prod_{i=1}^{s}v_{p_{i}}^{n_{i}}\cdot\prod_{i=1}^{s^{\prime}}v_{p_{i}^{\prime}}^{n_{i}^{\prime}}$,
but that is $X_{a\cdot b}$ by definition.
\end{proof}

\paragraph*{Notes on $\alpha_{aa}$}

As mentioned earlier, when constructing bases in terms of Lyndon words
for quantum shuffle algebras and quantum group algebras, we only use
$\sigma$ and not $\sigma^{-1}$.
\begin{enumerate}
\item If all $Q_{i}\in k$ are different from unity or indeterminate, then
$\left[n_{j}\right]_{Q_{j}}!$ is the Mahonian number such that
\[
\left[n\right]_{q}!=\prod_{j=1}^{n}\sum_{i=0}^{j-1}q^{i}=\prod_{j=1}^{n}\frac{1-q^{j}}{1-q}
\]
\item If some $Q_{i}\in k$ are unity, then we can't use the geometric series
formula for Mahonian numbers $\left[n_{j}\right]_{Q_{j}}!$ , and
use the expanded form instead:
\[
\left[n\right]_{q}!=\prod_{j=1}^{n}\sum_{i=0}^{j-1}q^{i}=\prod_{j=1}^{n}j=n!
\]
\end{enumerate}

\paragraph*{Special alternative case}

If we (1) use both $\sigma$ and $\sigma^{-1}$, (2) consider both
permutations and braidings in the sense of circular permutations and
braidings and not in the sense of linear arrangements, (3) continue
use the reduced expressions for permutations, (4) use symmetry considerations
in case of an ambiguity, and (5) set $Q_{i}^{2}=\prod_{k,l\in\left[\left[1,l_{i}\right]\right]}q_{x_{k}x_{l}}$,
then
\begin{enumerate}
\item If all $Q_{i}\in k$ are different from a square root of unity unity
or indeterminate, then $\left[n_{j}\right]_{Q_{j}}!$ has the standard
quantum group algebra definition
\[
\left[n\right]_{q}!=\prod_{i=1}^{n}\left[i\right]_{q}=\prod_{i=1}^{n}\frac{q^{i}-q^{-i}}{q-q^{-1}}=\prod_{i=1}^{n}\sum_{m\in M_{i}}q_{j}^{m}
\]
where $M_{i}=\left\{ m|-i+1\le m\le i-1\wedge m+i-1\in2\mathbb{Z}\right\} $,
i.e. $M_{i}=\left\{ -i+1,-i+3,\cdots,i-1\right\} $.
\item If some $Q_{i}\in k$ are square roots of unity, then we can't use
the formula $\left[i\right]_{q}=\frac{q^{i}-q^{-i}}{q-q^{-1}}$, thus
we have
\[
\left[n\right]_{q}!=\prod_{i=1}^{n}\sum_{m\in M_{i}}q_{j}^{m}
\]
\end{enumerate}

\section{Implementation in Mathematica}

We have implemented a Wolfram Mathematica package with functions for
quantum shuffle algebra and bases construction in terms of Lyndon
words. Since using a generic totally ordered set $X$ would have added
a layer of abstraction that is not necessarily essential for our purposes,
we have used $S=\left(\mathbb{N}\right)$ for convenience.

\subsection{Function listings}

The package functions include:
\begin{enumerate}
\item Comparison of two elements $a,b\in S$ using the total ordering on
$S$;
\begin{lstlisting}[language=Mathematica,basicstyle={\scriptsize\ttfamily},breaklines=true]
QSARelation[x_List, y_List] := (
   Module[{rel = 0, i}, (
     For[i = 1, i <= Min[Length[x], Length[y]], i++,
      If[rel == 0 && x[[i]] < y[[i]], rel = -1];
      If[rel == 0 && x[[i]] > y[[i]], rel = 1];
      ];
     If[rel == 0 && Length[x] < Length[y], rel = 1];
     If[rel == 0 && Length[x] > Length[y], rel = -1];
     rel
     )]);
\end{lstlisting}
\item Test of whether $v_{a}$ ($a\in S$) is prime;
\begin{lstlisting}[language=Mathematica,basicstyle={\scriptsize\ttfamily},breaklines=true]
QSAIsPrime[x_List] := (
   Module[{ans = True, i1}, (
     If[Length[x] > 1, {
       For[i1 = 2, i1 <= Length[x], i1++, {
          If[QSARelation[x[[i1 ;; -1]], x] > -1, ans = False]
          }];
       }];
     ans
     )]);
\end{lstlisting}
\item Finding the first prime of $v_{a}$ ($a\in S$) in its unique prime
factorization;
\begin{lstlisting}[language=Mathematica,basicstyle={\scriptsize\ttfamily},breaklines=true]
QSAFirstPrime[x_List] := (
   Module[{ans2 = {}, i2 = 1}, (
     While[
      i2 <= Length[x] && 
       QSAIsPrime[x[[1 ;; i2]]], {ans2 = x[[1 ;; i2]]; i2++;}];
     ans2
     )]);
\end{lstlisting}
\item Unique prime factorization (UPF) of $v_{a}$ ($a\in S$);
\begin{lstlisting}[language=Mathematica,basicstyle={\scriptsize\ttfamily},breaklines=true]
QSAUniquePrimeFactorization[x_List] := (
   Module[{remx = x, upf = {}, nextprime}, (
     While[Length[remx] > 0, {
       nextprime = QSAFirstPrime[remx];
       AppendTo[upf, nextprime];
       If[Length[nextprime] < Length[remx], 
        remx = remx[[Length[nextprime] + 1 ;; -1]], remx = {}];
       }];
     upf
     )]);
\end{lstlisting}
\item Partition, permutation, and other auxiliary functions;
\begin{lstlisting}[language=Mathematica,basicstyle={\scriptsize\ttfamily},breaklines=true]
QSAV[x_List, xi_] := (
   Apply[TensorProduct, Map[Subscript[v, xi[[#]]] &, x]]
   );
QSASumV[x_List, xi_] := (
   Module[{sum = 0}, (
     AddFun[item1_] := (
       If[ListQ[item1[[1]]], 
        sum = sum + item1[[2]] QSAV[item1[[1]], xi], 
        sum = sum + QSAV[item1, xi]]
       );
     Map[AddFun, x];
     sum
     )]);
QSAIndexToObject[x_, xi_] := (
   If[ListQ[x[[1]]], Map[xi[[#]] &, x[[;; , 1]]], Map[xi[[#]] &, x]]
   );
QSALengthToPartitionIndex[L_List] := (
   Module[{M = {1}, CurM = 1, i3}, (
     For[i3 = 1, i3 <= Length[L], i3++, {
       CurM = CurM + L[[i3]];
       AppendTo[M, CurM];
       }];
     M
     )]);
Permutation[xlist_List, yitem_, aword_, qletter_: q] := (
   Module[{plist, k, i}, (
     plist = {};
     For[k = 1, k <= Length[xlist], k++, {
       For[i = 1, i <= Length[xlist[[k, 1]]] + 1, i++, {
         AppendTo[
          plist, {Insert[xlist[[k, 1]], yitem, i], 
           xlist[[k, 2]] Product[
             If[qletter === 1, 1, Subscript[qletter, aword[[jj]], 
              aword[[Length[xlist[[k, 1]]] + 1]]]], {jj, i, 
              Length[xlist[[k, 1]]]}]}]
         }]
       }];
     plist
     )]);
PermuteList[zlist_List, aword_, qletter_: q] := (
   Module[{plist2, i1}, (
     plist2 = {};
     If[Length[zlist] == 1, plist2 = {{zlist, 1}}, {
       plist2 = {{{zlist[[1]]}, 1}};
       For[i1 = 2, i1 <= Length[zlist], i1++, {
         plist2 = Permutation[plist2, zlist[[i1]], aword, qletter]
         }];
       }];
     plist2
     )]);
Complement1[list1_List, list2_List] := (
   Module[{list3, list4 = {}, i}, (
     list3 = Complement[list1[[;; , 1]], list2[[;; , 1]]];
     For[i = 1, i <= Length[list1], i++, {
       If[MemberQ[list3, list1[[i, 1]]], AppendTo[list4, list1[[i]]]]
       }];
     list4
     )]);
QSAWordToLength[x_List] := (
   Map[Length[#] &, x]
   );
\end{lstlisting}
\item Quantum shuffle multiplication $v_{a}\cdot v_{b}$ ($a,b\in S$);
\begin{lstlisting}[language=Mathematica,basicstyle={\scriptsize\ttfamily},breaklines=true]
QSAGeneratePrimaryShuffles[L_List, aword_, qpar_: True] := (
   Module[{ShuffleList = {}, LL = Total[L], 
     M = QSALengthToPartitionIndex[L], ShuffleListTemp = {}, i4}, (
     For[i4 = 1, i4 < Length[M], i4++, {
       AppendTo[ShuffleListTemp, Range[M[[i4]], M[[i4 + 1]] - 1]];
       }];
     ShuffleList = 
      PermuteList[ShuffleListTemp, Range[Length[ShuffleListTemp]], 
       If[qpar, Q, 1]];
     For[i4 = 1, i4 <= Length[ShuffleList], i4++, {
       ShuffleList[[i4]] = {Flatten[ShuffleList[[i4, 1]]], 
          ShuffleList[[i4, 2]]};
       }];
     ShuffleList
     )]);
QSAGenerateShuffles[L_List, aword_, qpar_: True] := (
   Module[{ShuffleList = {}, LL = Total[L], 
     M = QSALengthToPartitionIndex[L], i4, item1, plist0}, (
     ShuffleListTemp = PermuteList[Range[LL], aword, If[qpar, q, 1]];
     AddShuffle[item_] := (
       TestF = True;
       item1 = item[[1]];
       For[i4 = 1, i4 < Length[M], i4++, {
         If[
           Not[OrderedQ[
             Select[item1, # >= M[[i4]] && # < M[[i4 + 1]] &]]], 
           TestF = False];
         }];
       If[TestF, AppendTo[ShuffleList, item]];
       );
     Map[AddShuffle, ShuffleListTemp];
     plist0 = QSAGeneratePrimaryShuffles[L, aword, qpar];
     ShuffleList = Union[plist0, Complement1[ShuffleList, plist0]];
     ShuffleList
     )]);
QSAGenerateSecondaryShuffles[L_List, aword_, qpar_: True] := (
   Complement1[QSAGenerateShuffles[L, aword, qpar], 
    QSAGeneratePrimaryShuffles[L, aword, qpar]]
   );
QSAShuffleMultiplication[x_List] := 
  QSASumV[QSAGenerateShuffles[QSAWordToLength[x], Flatten[x]], 
   Flatten[x]];
\end{lstlisting}
\pagebreak{}
\item Calculation of $X_{a}$ ($a\in S$);
\begin{lstlisting}[language=Mathematica,basicstyle={\scriptsize\ttfamily},breaklines=true]
QSAX[x_, qpar_: True] := 
  QSASumV[QSAGenerateShuffles[
    QSAWordToLength[QSAUniquePrimeFactorization[x]], x, qpar], x];
\end{lstlisting}
\item Expression of $v_{a}$ ($a\in S$) in terms of Lyndon words (primes)
.
\begin{lstlisting}[language=Mathematica,basicstyle={\scriptsize\ttfamily},breaklines=true]
QSAExpressInLyndonWords[x_] := (
   Module[{shuffles, shuffles1, lhs = {}, rhs = {}, eqns = {}, 
     rhs1 = {}, ia, avec, coeffarrays, sol}, (
     shuffles = QSAGenerateShuffles[L, x];
     shuffles1 = QSAIndexToObject[shuffles, a];
     For[ia = 1, ia <= Length[shuffles1], ia++, {
       AppendTo[lhs, Evaluate[Subscript[X, shuffles1[[ia]]]]];
       AppendTo[rhs, QSAX[shuffles1[[ia]]]];
       AppendTo[rhs1, QSAX[shuffles1[[ia]], False]];
       AppendTo[eqns, 
        Evaluate[Subscript[X, shuffles1[[ia]]]] == 
         QSAX[shuffles1[[ia]]]];
       }];
     avec = DeleteDuplicates[Flatten[rhs1, 1, Plus]];
     coeffarrays = Normal[CoefficientArrays[eqns, avec]];
     sol = LinearSolve[coeffarrays[[2]], coeffarrays[[1]]];
     MapThread[#1 == #2 &, {avec, sol}]
     )]);
\end{lstlisting}
\end{enumerate}

\subsection{Calculation examples}
\begin{example}[Unique prime factorization]

Consider $a=18\cdot19\cdot4\cdot8\cdot5\cdot7$. It is written in
Mathematica notation using an ordered list as $a=\{18,19,4,8,5,7\}$.
(Since in this section we are specifically speaking of the Mathematica
package, we will disregard the usual notation of curly brackets for
an unordered set and use it only for an ordered list in this section.)

Using the function $\mathtt{QSAUniquePrimeFactorization}$, we obtain
the UPF as an ordered list of two ordered lists $\left\{ \left\{ 18\right\} ,\left\{ 19,4,8,5,7\right\} \right\} $,
which means that $a=p_{1}\cdot p_{2}$, where primes $p_{1}$and $p_{2}$
are $p_{1}=18$ and $p_{2}=19\cdot4\cdot8\cdot5\cdot7$.
\end{example}

\begin{example}[Calculation of $X_{a}$ ($a\in S$)]
\textcolor{white}{\scriptsize{}T}{\scriptsize\par}

Consider again $a=\{18,19,4,8,5,7\}$ in the Mathematica notation
for an ordered list. Using the function $\mathtt{QSAX}$, we obtain
for $X_{a}$ that 

\begin{multline*}
X_{a}=v_{18}\otimes v_{19}\otimes v_{4}\otimes v_{8}\otimes v_{5}\otimes v_{7}+\\
+Q_{1,2}v_{19}\otimes v_{4}\otimes v_{8}\otimes v_{5}\otimes v_{7}\otimes v_{18}+\\
+q_{4,8}q_{8,5}q_{18,19}q_{19,4}v_{19}\otimes v_{4}\otimes v_{8}\otimes v_{5}\otimes v_{18}\otimes v_{7}+\\
+q_{4,8}q_{18,19}q_{19,4}v_{19}\otimes v_{4}\otimes v_{8}\otimes v_{18}\otimes v_{5}\otimes v_{7}+\\
+q_{18,19}q_{19,4}v_{19}\otimes v_{4}\otimes v_{18}\otimes v_{8}\otimes v_{5}\otimes v_{7}+\\
+q_{18,19}v_{19}\otimes v_{18}\otimes v_{4}\otimes v_{8}\otimes v_{5}\otimes v_{7}
\end{multline*}
where $\sigma\left(v_{i}\otimes v_{j}\right)=q_{ij}\left(v_{j}\otimes v_{i}\right)$
for $i,j\in\mathbb{N}$ and, for conciseness purposes, $\sigma\left(v_{i}\otimes v_{j}\right)=Q_{ij}\left(v_{j}\otimes v_{i}\right)$
for $i,j\in P\subset S=\left(\mathbb{N}\right)$.

Using the function $\mathtt{QSAPrimaryCoefficient}$, we obtain $\alpha_{aa}=1$.
\end{example}

\begin{example}[Quantum shuffle multiplication]
\label{exa:QSM}

Consider $b=\left\{ 5,10,10\right\} $ and $c=\left\{ 7,4,10\right\} $.
Using the function $\mathtt{QSAShuffleMultiplication}$, we obtain
for quantum shuffle multiplication $v_{b}\cdot v_{c}$ the result
in Figure \ref{fig:QSMR}.

\begin{figure}
\includegraphics[angle=90,height=0.9\textheight]{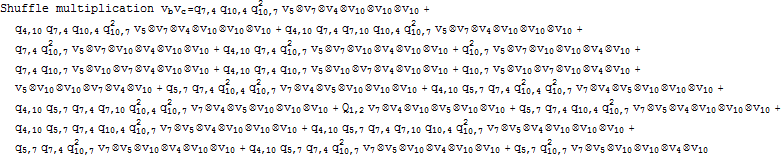}\caption{\label{fig:QSMR}Result in Mathematica for quantum shuffle multiplication
$v_{b}\cdot v_{c}$ in Example \ref{exa:QSM}.}
\end{figure}
\end{example}

\begin{example}[Expression of $v_{a}$ ($a\in S$) in terms of $X_{c}$'s ($c\in S$
and $c\ge a$)]
\label{exa:EITLW}

We again consider $a=\{18,19,4,8,5,7\}$. Solving the matrix equation,
we express $v_{a}$ and other summands in $X_{a}$ in terms of $X_{c}$'s
($c\ge a$). The result is shown in Figure \ref{fig:QSMR-1}.

\begin{figure}
\includegraphics[angle=90,height=0.9\textheight]{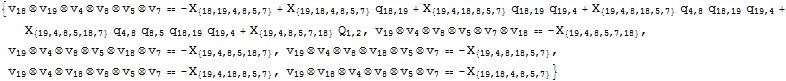}\caption{\label{fig:QSMR-1}Result in Mathematica for expression of $v_{a}$
in terms of Lyndon words in Example \ref{exa:EITLW}.}
\end{figure}
\end{example}

\section{Quantum group algebras II}

\subsection{Idea and underlying principles}

We can adapt the construction of bases of quantum shuffle algebras
in terms of Lyndon words to quantum group algebras by noting the following:
\begin{enumerate}
\item Braiding is defined on $T\left(V\right)$ for a $U_{q}\left(\mathfrak{g}\right)$-module
$V$ by the universal R-matrix and on $S_{\sigma}\left(M^{R}\right)$
by an $N\times N$ matrix $\left(q_{ij}\right)$.
\item Specifying a $U_{q}\left(\mathfrak{g}\right)$-module $T\left(V\right)$,
where $V=\mathfrak{g}$, we can define an associative structure on
$T\left(V\right)$, for $x_{1},\cdots,x_{n}$ in $V$,
\[
\left(x_{1}\otimes\cdots\otimes x_{p}\right)\cdot\left(x_{p+1}\otimes\cdots\otimes x_{n}\right)=\sum_{w\in\sum_{p,n-p}}T_{w}\left(x_{1}\otimes\cdots\otimes x_{n}\right)
\]
Proof of the associativity is the same as in case of quantum shuffle
algebras.
\item If case the quantum group parameter $q$ is unity, the R-matrix is
$R=1\otimes1$, and braiding is trivial: $\sigma\left(v\otimes w\right)=w\otimes v$.
\item One should keep in mind that generators of $U_{q}\left(\mathfrak{g}\right)$
are not the linear basis of $V$.
\item Failing to adapt the bases construction method for quantum group algebras
as a whole, we will use the same idea and principles for subalgebras
of quantum group algebras.
\end{enumerate}

\subsection{Braiding and the universal $R$-matrix}

If $v\in V$ and $w\in W$, then 
\[
\sigma\left(v\otimes w\right)=\tau\left(R\left(v\otimes w\right)\right)
\]
where $\tau$ is the flip $\tau:v\otimes w\mapsto w\otimes v$ and
$R$ is the universal $R$-matrix of the form \cite[p. 175]{kassel}
\[
R=\sum_{i}s_{i}\otimes t_{i}
\]
i.e.
\[
\sigma\left(v\otimes w\right)=\sum_{i}t_{i}w\otimes s_{i}v
\]

\begin{prop}[Block permutation of two equal primes]

Let $a\in S$ and $a=p^{2}$ its prime factorization. We denote the
string decomposition of $p$ as $\prod_{j=1}^{L}e_{j}$, where $L$
is its length. Let $w$ be in $\sum_{2}$, a subgroup of ``block
permutations'' that permutes the $p$'s among themselves. Let the
braiding in $V\otimes V$ be given by $\sigma\left(e_{i}\otimes e_{j}\right)=\tau\left(R\left(e_{i}\otimes e_{j}\right)\right)$,
where $R$ is the universal $R$-matrix. Then either

\begin{enumerate}
\item $w=Id$ and $T_{w}\left(v_{a}\right)=v_{a}$;
\item $w=s_{1}$, $T_{w}=\sigma_{1}=\left(1\mapsto2\right)$, and 
\[
T_{w}\left(v_{a}\right)=\left(\sum_{i_{1},\cdots,i_{L}}t_{i_{1}}\cdots t_{i_{L}}\right)v_{p}\otimes\left(\sum_{i_{1},\cdots,i_{L}}s_{i_{1}}\cdots s_{i_{L}}\right)v_{p}\mathrm{;}
\]
\item $w=s_{1}$ and $T_{w}=\sigma_{1}^{-1}=\left(2\mapsto1\right)$ (N/A).
\end{enumerate}
\end{prop}

\begin{proof}
This is trivial application of the expression of the universal $R$-matrix
in terms of elements $s_{i}$ and $t_{i}$ to the framework we have
developed above.
\end{proof}
To apply the same method of construction of bases in terms of Lyndon
words for the quantum group algebra as for quantum shuffle algebras,
we need the universal $R$-matrix to be diagonal, which corresponds
to diagonal braiding. We will assume that this is the case. For that,
it is sufficient that all $t_{i}$ and $s_{i}$ in the expression
for the universal $R$-matrix are elements of the Cartan subalgebra.

For a highest weight module of a quantum group algebra, it is useful
to write the weight vector expression as
\[
k_{\lambda}\cdot v=d_{\lambda}v=c_{\lambda}q^{\left(\lambda,v\right)}v
\]
where $\nu$ is an element of the weight lattice. \cite[p. 72]{jantzen}

Accordingly, we consider the specific form of the diagonal universal
$R$-matrix
\[
R=q^{\sum_{i}d_{ij}\tilde{k}_{\lambda_{i}}\otimes\tilde{k}_{\mu_{i}}}
\]
where Cartan subalgebra generators $k_{i}$ are formally identified
with $q^{\tilde{k}_{i}}$ \cite{khoroshin}. It is one of the possible
functions $f$ as described, for example, in \cite[ch. 3,7]{jantzen}.
Setting $q^{\tilde{k}_{\lambda_{i}}}e_{j}=q^{\alpha_{ij}}e_{j}$,
and $q^{\tilde{k}_{\mu_{i}}}e_{j}=q^{\beta_{ij}}e_{j}$, it acts on
an element $v\otimes w\in V\otimes W$ as 
\[
R\cdot\left(v\otimes w\right)=q^{\sum_{k}d_{ij}\alpha_{ki}\beta_{kj}}\left(v\otimes w\right)
\]

We now state the following trivial propositions:
\begin{prop}[Diagonal braiding on $V\otimes V$]
\textcolor{white}{\scriptsize{}T}

\begin{enumerate}
\item If all elements $t_{i}$ and $s_{i}$ in the expression {\small{}$R=\sum_{i}s_{i}\otimes t_{i}$
of the universal $R$-matrix are elements of the Cartan subalgebra,
then braiding in $V\otimes V$ is diagonal (i.e. given by $\sigma\left(e_{i}\otimes e_{j}\right)=q_{ij}\left(e_{j}\otimes e_{i}\right)$).}{\small\par}
\item {\small{}If additionally $V$ is a highest weight module, $q^{\tilde{k}_{\lambda_{i}}}e_{j}=q^{\alpha_{ij}}e_{j}$,
and $q^{\tilde{k}_{\mu_{i}}}e_{j}=q^{\beta_{ij}}e_{j}$ for all $i$
and $j$, then
\[
\sigma\left(e_{i}\otimes e_{j}\right)=q^{\sum_{k}d_{ij}\alpha_{ki}\beta_{kj}}\left(e_{j}\otimes e_{i}\right)
\]
}{\small\par}
\end{enumerate}
\end{prop}

\subsection{Bases in terms of Lyndon words}
\begin{prop}[Block permutation of two equal primes in case of a diagonal universal
$R$-matrix ]

Let $a\in S$ and $a=p^{2}$ its prime factorization. We denote the
string decomposition of $p$ as $\prod_{j=1}^{L}e_{j}$, where $L$
is its length. Let $w$ be in $\sum_{2}$, a subgroup of ``block
permutations'' that permutes the $p$'s among themselves. Let the
braiding in $V\otimes V$ be given by $\sigma\left(e_{i}\otimes e_{j}\right)=q_{ij}\left(e_{j}\otimes e_{i}\right)$,
and let $Q=\prod_{k,l\in\left[\left[1,L\right]\right]}q_{kl}$. Then
either

\begin{enumerate}
\item $w=Id$ and $T_{w}\left(v_{a}\right)=v_{a}$;
\item $w=s_{1}$, $T_{w}=\sigma_{1}=\left(1\mapsto2\right)$, and $T_{w}\left(v_{a}\right)=Qv_{a}$;
\item $w=s_{1}$, $T_{w}=\sigma_{1}^{-1}=\left(2\mapsto1\right)$, and $T_{w}\left(v_{a}\right)=Q^{-1}v_{a}$
(N/A).
\end{enumerate}
\end{prop}

Assuming that the universal $R$-matrix is diagonal (i.e. we have
diagonal braiding) for quantum group algebra (or its subalgebra),
we can apply Proposition \ref{prop:BasesQSA} used for quantum shuffle
algebras to the quantum group algebra (resp. its subalgebra) verbatim.
Even though the proposition's formulation is the same as that of Proposition
\ref{prop:BasesQSA} (as can be expected), will restate it here for
text structure and reference purposes.
\begin{prop}[Bases of quantum group algebras in terms of Lyndon words]
\label{prop:BasesQGA}

Let $a\in S\left(u\right)$ and $a=\prod_{i=1}^{s}p_{i}^{n_{i}}$
be its unique prime factorization. We define $X_{a}=\prod_{i=1}^{s}v_{p_{i}}^{n_{i}}$,
where quantum multiplication is used between the terms of the form
$v_{p_{i}}$. Then:

\begin{enumerate}
\item $X_{a},a\in S\left(u\right)$ form a basis of $T\left(V\right)\left(u\right)$;
and
\item the change of basis with respect to $a$ is triangular, i.e. there
exist $\alpha_{ab}\in k$ such that $X_{a}=\sum_{a\le b}\alpha_{ab}v_{b}$;
\item setting for $1\le i\le s$ $p_{i}=\prod_{j=1}^{l_{i}}x_{ij}$ and
$Q_{i}=\prod_{k,l\in\left[\left[1,l_{i}\right]\right]}q_{x_{k}x_{l}}$,
we have $\alpha_{aa}=\prod_{j=1}^{s}\left(\left[n_{j}\right]_{Q_{j}}!\right)\ne0$
for all $Q_{i}\in k$ being different from unity, or $Q_{i}$'s indeterminate;
\item the $X_{a}$'s, $a\in S$, form a linear basis of $T\left(V\right)$;
\item the $v_{p}$'s, $p\in P$, form a polynomial basis for $T\left(V\right)$;
\item $X_{a}\cdot X_{b}=X_{a\cdot b}$ for $a,b\in S$.
\end{enumerate}
\end{prop}

\paragraph*{Notes on $\alpha_{aa}$}

As mentioned earlier, when constructing bases in terms of Lyndon words
for quantum shuffle algebras and quantum group algebras, we only deal
with $\sigma$ and not $\sigma^{-1}$.
\begin{enumerate}
\item If all $Q_{i}\in k$ are different from unity or indeterminate, then
$\left[n_{j}\right]_{Q_{j}}!$ is the Mahonian number such that
\[
\left[n\right]_{r}!=\prod_{j=1}^{n}\sum_{i=0}^{j-1}r^{i}=\prod_{j=1}^{n}\frac{1-r^{j}}{1-r}
\]
\item If some $Q_{i}\in k$ are unity, then can't use the geometric series
formula and have for respective Mahonian numbers $\left[n_{j}\right]_{Q_{j}}!$
that
\[
\left[n\right]_{r}!=\prod_{j=1}^{n}\sum_{i=0}^{j-1}r^{i}=\prod_{j=1}^{n}j=n!
\]
\item If the quantum group parameter $q$ is unity, then we have a nondeformed
universal enveloping algebra module, and all $Q_{i}$'s are unity.
\item The case of $q$ being a root of unity (other than $q=1$) has to
be addressed separately.
\end{enumerate}

\paragraph*{Algorithm for bases construction in terms of Lyndon words}

Similarly to classical quantum algebras and the general case of quantum
group algebras, Proposition \ref{prop:BasesQGA} implicitly gives
a method of bases construction in terms of Lyndon words. In Example
\ref{exa:EITLW}, we have expressed one $v_{a},a\in S$ in terms of
$X_{c}$'s, which is equivalent to expression as a polynomial of $v_{p}$'s,
where $p\in P$. Suppose that we have a linear basis of $T\left(V\right)$
or its subspace as a set of $v_{a}$'s, where $a\in S$. Then we can
express that basis using a polynomial basis of Lyndon words as follows:
(1) following Proposition \ref{prop:BasesQGA}, express each $v_{a}$
in the linear basis of $T\left(V\right)$ using shuffle multiplication
in terms of primes $v_{p},p\in P$; and (2) take the union of the
sets of applicable $v_{p}$'s, remembering to delete any duplicates
during the process.

\paragraph*{$U_{q}\left(\mathfrak{g}\right)$ structure and quantum shuffle multiplication}

The quantum shuffle multiplication is based on the natural representation
of the braid group on quantum group algebra on $T\left(V\right)$,
where $V$ is a quantum group algebra module. It is an additional
structure that is, as a multiplication, not compatible with the quantum
group algebra's coproduct. If we wish, can can define an additional
coproduct on the quantum group algebra that is compatible with the
quantum shuffle multiplication -- for example, by using the approach
of universal construction in the braid category.

\paragraph*{Application of the Mathematica function package}

Continuing to consider the case where the braiding is given in $V\otimes V$
by $\sigma\left(e_{i}\otimes e_{j}\right)=q_{ij}\left(e_{j}\otimes e_{i}\right)$,
where $V$ is a quantum group algebra module, we can directly apply
the Mathematica program discussed above to quantum shuffle multiplication
and construction of bases in terms of Lyndon words for quantum group
algebras.

\paragraph*{Scope of applicability to quantum group algebras}

The diagonal universal $R$-matrix condition is quite restrictive,
and we expect it to apply only to exceptional types of quantum group
algebras. This can be seen from the expression for the standard universal
$R$-matrix for $U_{q}\left(\mathfrak{g}\right)$ with assumption
that $\mathfrak{g}$ is of finite type:
\[
R_{q}=exp\left(q\sum_{i,j}\left(B^{-1}\right)_{ij}k_{i}\otimes k_{j}\right)\prod_{\beta}exp_{q_{\beta}}\left[\left(1-q_{\beta}^{-2}\right)e_{\beta}\otimes f_{\beta}\right]
\]
where the product is over all the positive roots of $\mathfrak{g}$,
and the order of the terms is such that $\beta_{r}$ appears to the
left of $\beta_{s}$ if $r\ge s$. \cite[Theorem 8.3.9]{charipressley}

\paragraph*{Applicability to positive and negative parts of quantum group algebras}

Continuing to assume that $\mathfrak{g}$ is if finite type, the positive
part $U_{q}^{+}\left(\mathfrak{g}\right)$ of the quantum group algebra
is a Nichols algebra. Restriction of the universal $R$-matrix of
$U_{q}\left(\mathfrak{g}\right)$ to $U_{q}^{+}\left(\mathfrak{g}\right)$
satisfies the diagonality condition, i.e. the braiding in $U_{q}^{+}\left(\mathfrak{g}\right)$
is diagonal \cite{dumas,takeuchi}. Fundamental results on this include
\cite{MRosso} and, in form an implicit discussion, \cite{lusztig}.
This diagonal braiding is given by $q_{ij}=q^{d_{i}a_{ij}}$ \cite{MRosso,dumas,takeuchi}.
Moreover, keeping in mind that the universal $R$-matrix and the corresponding
braiding are not unique, the results of constructing bases of $U_{q}^{+}\left(\mathfrak{g}\right)$
in terms of Lyndon words as described in this section and in \cite{MRosso}
are equivalent. Same applies to the negative part $U_{q}^{-}\left(\mathfrak{g}\right)$.

\paragraph*{Specialized and non-specialized quantum group algebras}

The non-restricted specialization of $U_{q}\left(\mathfrak{g}\right)$
is obtained by using a specific value of $q$ instead of the indeterminate.
For $q$ not being a root of unity, the restricted and non-restricted
specializations coincide. While the standard universal $R$-matrix
obtained for a non-specialized $U_{q}\left(\mathfrak{g}\right)$ using
\cite[Theorem 8.3.9]{charipressley} gives a well-defined endomorphism
for the specialized case when $q$ is not a root of unity, it does
not give an element of $U_{q}\otimes U_{q}$ due to fractional powers
of $q$ and due to the expression being an infinite sum. This issue
is not substantial for our purposes and can be addressed as described
in \cite[pp. 327-331]{charipressley}.

\subsection{A PBW-type basis analogy for $T\left(V\right)$}

Parts (4) and (5) of Proposition \ref{prop:BasesQGA} can be restated
to resemble a PBW-type bases theorem.
\begin{defn}[$\left(T^{j}\right)^{i}$ notation]

Define $\left(T^{j}\right)^{i}$ as $\left(j\mapsto j-i\right)\in B_{n}$
that permutes the $p\in P$ (or respective $v_{p}\in T\left(V\right)$)
from position $j$ in $a$ ($v_{a}$ resp.) to position $j-i$. In
the same context, let $I=\left[\left[1,n\right]\right]$.
\end{defn}

\begin{cor}[A PBW-type basis analogy for $T\left(V\right)$]
\textcolor{white}{\scriptsize{}T}

\begin{enumerate}
\item The set of elements of $T\left(V\right)\left(u\right)$ of the form
$T^{\gamma_{1}}\cdots T^{\gamma_{k}}v_{a}$ where $a\in S\left(u\right)$
, with $\gamma_{j}$ being a monotonically increasing sequence of
$k$ elements of $I$, i.e.
\[
\gamma_{1}\le\cdots\le\gamma_{k}
\]
and with $k$ any non-negative integer, is a linear basis of $T\left(V\right)\left(u\right)$.
\item The set of elements of $T\left(V\right)$ of the form $T^{\gamma_{1}}\cdots T^{\gamma_{k}}v_{p}$
where $p\in P$, with $\gamma_{j}$ being a monotonically increasing
sequence of $k$ elements of $I$, i.e.
\[
\gamma_{1}\le\cdots\le\gamma_{k}
\]
and with $k$ any non-negative integer, is a polynomial basis of $T\left(V\right)$.
\end{enumerate}
\end{cor}

\subsection{Case of $q$ as a root of unity}

\subsubsection{Restricted and non-restricted specializations}

For a quantum group $U_{q}\left(\mathfrak{g}\right)$, there are two
ways to specialize the group parameter $q$ to a root of unity $\epsilon$:
non-restricted and restricted, resulting in different algebras $U_{\epsilon}\left(q\right)$
and $U_{\epsilon}^{res}\left(q\right)$ respectively. It is usually
assumed that $\epsilon$ is the primitive $l$th root of unity, where
$l$ is odd and $l>d_{i}$ for all $i$ ($d_{i}$ are the coprime
positive integers such that the matrix $\left(d_{i}a_{ij}\right)$
is symmetric). Both $U_{\epsilon}\left(\mathfrak{g}\right)$ and $U_{\epsilon}^{res}\left(\mathfrak{g}\right)$
are not quasitriangular and do not have a universal $R$-matrix. \cite[p. 327-329]{charipressley}
\begin{enumerate}
\item However, it is possible to obtain matrix-valued solutions of the quantum
Yang-Baxter equation (QYBE) on representations of $U_{\epsilon}\left(\mathfrak{g}\right)$,
where the tensor product is commutative up to an isomorphism. In this
subcase, we can directly apply the method of constructing bases in
terms of Lyndon words to the solution of the QYBE.
\item This method of obtaining matrix-valued solutions of QYBE is specific
to $U_{\epsilon}\left(\mathfrak{g}\right)$ and is not applicable
to $U_{\epsilon}^{res}\left(\mathfrak{g}\right)$.
\end{enumerate}

\subsubsection{Solutions of QYBE in $U_{\epsilon}\left(\mathfrak{g}\right)$}

The following proposition is useful for obtaining matrix-valued QYBE
solutions in the non-restricted case of $U_{\epsilon}\left(\mathfrak{g}\right)$.
\begin{prop}[{Solutions of QYBE in $U_{\epsilon}\left(\mathfrak{g}\right)$ \cite[pp. 349-359]{charipressley}}]

Let $\left\{ V\left(u\right)\right\} _{u\in V}$ be a family of representations
of a Hopf algebra $A$, all with the same underlying vector space
$V$ and parametrized by the elements $v$ of some set $\mathcal{V}$,
such that:

for all $v_{1},v_{2}\in\mathcal{V}$, there is an isomorphism of representations
\[
I\left(v_{1},v_{2}\right):V\left(v_{1}\right)\otimes V\left(v_{2}\right)\rightarrow V\left(v_{2}\right)\otimes V\left(v_{1}\right)
\]

for all $v_{1},v_{2},v_{3}\in\mathcal{V}$, the only isomorphisms
of representations
\[
V\left(v_{1}\right)\otimes V\left(v_{2}\right)\otimes V\left(v_{3}\right)\rightarrow V\left(v_{1}\right)\otimes V\left(v_{2}\right)\otimes V\left(v_{3}\right)
\]
are the scalar multiples of identity.

Then, if $R=\tau\circ I$, where $\tau$ is the interchange of the
factors in the tensor product,
\[
R_{12}\left(v_{1},v_{2}\right)R_{13}\left(v_{1},v_{3}\right)R_{23}\left(v_{2},v_{3}\right)=cR_{23}\left(v_{2},v_{3}\right)R_{13}\left(v_{1},v_{3}\right)R_{12}\left(v_{1},v_{2}\right)
\]
where $c$ is a scalar (possibly depending on $v_{1}$, $v_{2}$,
and $v_{3}$).
\end{prop}

Since the intertwiners $I$ are only determined up to a scalar multiple,
it may be possible to normalize them so that $c=1$.

\section{Discussion}

We have derived a method to construct bases of positive (negative)
parts $U_{q}^{+}\left(\mathfrak{g}\right)$ ($U_{q}^{-}\left(\mathfrak{g}\right)$
resp.) of quantum group algebras using Lyndon words (primes). We have
examined the case of quantum parameter $q$ being a root of unity.
A secondary result is that we have developed a Wolfram Mathematica
package that performs a number of relevant operations, including quantum
shuffle multiplication and construction of bases in terms of Lyndon
words for quantum group algebras.

We have founded the bases construction method on classical shuffle
algebra \cite{Radford} and quantum shuffle algebra \cite{MRosso}
theory. In this work, we have attempted to balance independent perspective
and coherence with these primary references. We found that our end
result for quantum group algebras agrees with that in \cite{MRosso}.
On the one hand, this limits the novelty of our work, but on the other,
it validates it.

The Mathematica package's functionality is limited to the concrete
case of $X=\mathbb{N}$, but can be easily extended to the general
case of any totally ordered set. The memory requirement of its current
implementation is roughly proportional to the factorial of the length
of a word, and all calculations are done in random-access memory.
To address this issue, one can optimize the source code, perform the
calculations piecewise, and/or store interstitial calculation results
in a file.

Interesting directions for more specific research include (1) determining
whether our bases construction method may have broader applications
for specific types of Kac-Moody algebra $\mathfrak{g}$ than as detailed
here, (2) fully developing the approach toward diagonal braiding using
the 
\[
T_{w}\left(v_{a}\right)=\left(\sum_{i_{1},\cdots,i_{L}}t_{i_{1}}\cdots t_{i_{L}}\right)v_{p}\otimes\left(\sum_{i_{1},\cdots,i_{L}}s_{i_{1}}\cdots s_{i_{L}}\right)v_{p}
\]
 expression, and (3) researching the case in which $q$ is a root
of unity in more depth, including cyclic representations of $U_{\epsilon}\left(\mathfrak{g}\right)$.
With some extension of functionality, performance optimization, and
thorough documentation, the Mathematica package can be shared publicly
by means of a repository and accessed by practitioners.

This work details theory of shuffle algebras and bases construction
in terms of Lyndon words (primes) for classical and quantum shuffle
algebras. We have applied this theory to positive (negative) parts
of quantum group algebras, including the case of quantum parameter
$q$ being a root of unity.

This thesis (mémoire de stage), along with the Mathematica package,
may be useful to graduate students and researchers familiarizing themselves
with the topic. 

\bibliographystyle{amsplain}
\nocite{*}
\bibliography{C:/Users/ereme/OneDrive/Documents/UPMC/Stage}

\end{document}